\documentclass[11pt]{article}
\usepackage{amsmath,amssymb,amsthm}
\newcommand{\re}{\mathbb{R}}
\newcommand{\co}{\mathbb{C}}

\newcommand{\T}{\mathrm{T}}
\newcommand{\D}{\mathrm{D}}
\newcommand{\R}{\mathrm{R}}
\newcommand{\rp}{\mbox{Re}}

\long\def\symbolfootnote[#1]#2{\begingroup%
\def\thefootnote{\fnsymbol{footnote}}\footnote[#1]{#2}\endgroup}

\newtheorem{thm}{Theorem}[section]
\newtheorem{prop}[thm]{Proposition}
\newtheorem{lem}[thm]{Lemma}
\newtheorem{cor}[thm]{Corollary}

\theoremstyle{definition}

\newtheorem{defn}[thm]{Definition}

\theoremstyle{remark}

\newtheorem{rem}[thm]{Remark}
\usepackage{comment}

\title{On solvability of nonlinear partial differential systems of any order in the complex plane}

\author{Yifei Pan}

\begin{document}

\maketitle
\begin{center}
\emph{Dedicated to the memory of Tom Wolff}
\end{center}
\begin{abstract}
We prove a general existence theorem for nonlinear partial differential systems of any order in one complex variable.
A special case of first order contains a well-known theorem of Nijenhuis and Woolf concerning local existence of J-holomorphic curves
on almost complex manifolds. As an application to differential geometry, we prove that non constant harmonic map always exists locally from a Riemann
surface to a Riemannian manifold with a prescribed tangent vector at a given point.
\end{abstract}

\large

\section{Introduction}\label{sec0}\symbolfootnote[0]{MSC 2010: 35G20 (Primary); 32G05, 30G20 (Secondary)}

Ever since the famous example of H. Lewy [L], local solvability of partial differential systems of equations has been a fundamental problem, and there have been important works for examples by Nirenberg-Treves [NT], Beals-Fefferman [BF]and H\" ormander [H]. A series of more recent important  works have been done by Lerner [LE1] and Decker [D]. For an extensive account of the subject, we refer to Lerner's survey up to 2002 [LE2].

In this paper we consider the nonlinear problem for existence only in dimension two from the point view of complex analysis. Our idea is motivated from the study of J-holomorphic curves on an almost complex manifold. Since the fundamental work of M. Gromov, J-holomorphic curves have been extensively studied on almost complex manifolds and applied to Symplectic topology as an important tool. On the other hand, local existence of J-holomorphic curves on a almost complex manifold is guaranteed by a
well-known and classical theorem of Nijenhuis and Woolf [NW], which, in turn,  is equivalent to local existence of a special
form of first order differential
system in terms of Cauchy-Riemann operator of one complex variable.

The purpose of this paper is to prove a general local  and global existence theorem of partial differential system of any order
in one complex variable, which in
particular extends the theorem of Nijenhuis-Woolf. Our method is rather effective for Laplace operator. Namely, we prove that any system with a power of Laplace operator $\Delta^m$ as principle part can be always solvable locally for any jets of order $2m-1$ at the origin, both for complex valued solutions if system is complex or real valued solutions if the system is real valued (see Theorem 1.5, 1.6 and 1.7). In other words, we consider the following as one of the main results of this paper.

The first result is concerned with local existence.

\bigskip

\noindent
{\bf Theorem A.} {\sl Let $A=(A_1, A_2,..., A_N)$ be any function of class $C^{k}(k\geq 2)$ in its variables in $\re^M$ for some $M$. Let $p(z)$ be any polynomial of degree less or equal to $2m-1$. Then the system of $m$-Laplace, $u:\{z\in\co:|z|\leq R\}\to \re^N$,
$$\Delta^m u=A(z,u,\nabla u, ..., \nabla^{2m-1}u)$$
has solutions of class $C^{2m+k}$ in $\{z\in\co:|z|\leq R\}$ for sufficiently small values of $R$ and the solution $u=p(z)+O(|z|^{2m})$ near the origin.
}
\bigskip

When the system is autonomous, we can find global solutions.
\bigskip

\noindent
{\bf Theorem B.} {\sl Let $A=(A_1, A_2,..., A_N)$ be any function of class $C^{k}(k\geq 2)$ in its variables in $\re^M$ for some $M$. Assume that
$$A(0)=0, \mbox{ } \nabla A(0)=0.$$
 Then the autonomous system of $m$-Laplace, $u:\{z\in\co:|z|\leq R\}\to \re^N$,
$$\Delta^m u=A(u,\nabla u, ..., \nabla^{2m-1}u, \nabla^{2m}u)$$
has solutions of class $C^{2m+k}$ in $\{z\in\co:|z|\leq R\}$ for any given $R$ with vanishing order of $2m$ at the origin.
}

\bigskip

We remark that it is a classical that the equation $\Delta u=e^{2u}$ has no solutions in the whole plane $\co$ due to Ahlfors [A], and local existence puts constrain on the radius due to Osserman [O]. Therefore, the vanishing condition in Theorem B can't be dropped, and the existence of radius in Theorem A can't be arbitrarily large. As a simple example, we know the equation $\Delta n(z)+|\nabla n(z)|^2n(z)=0$ has a singular solution $n(z)=\frac{z}{|z|}$ with singularity at $z=0$. But by Theorem B, the equation has smooth solutions in any disk centered at the origin of vanishing order $2$ at the origin.

As an application to differential geometry, we prove a local existence of non-trivial harmonic map
from a Riemann surface to a Riemannian manifold. More precisely, we prove
\bigskip

\noindent
{\bf Theorem C.} {\sl
Let $S$ be a Riemann surface and $N$ be a Riemannian manifold of class $C^{3}$. Let $z_0\in S$, and $p\in N$ and $v\in T_pN$. Then there is a local harmonic
map $\phi$ of class $C^2$ from a neighborhood of $z_0$ to $N$ such that $\phi(z_0)=p$ and $d\phi(z_0)(\frac{\partial}{\partial x}) =v$.}
\bigskip

This result does not seem to have appeared in the literature given the fact that there have been extensive works on closed harmonic maps. This theorem allows one to define Kobayashi metric on the tangent bundle of any Riemannian manifold through local harmonic maps, which could potentially paly a similar role as J-holomorphic curves do on a almost complex manifold. The same result of Theorem B is also true for
bi-harmonic maps (but not harmonic) or m-harmonic maps. We will make some discussion in this paper, and will continue the study in hoping to obtain more significant results in a future paper.

Finally we state the existence of ordinary differential system in $\bar\partial$ in one complex variable.

\bigskip

\noindent
{\bf Theorem D.} {\sl Let $F(z, \eta_0, \eta_1,...,\eta_{m-1})=(F^1, ..., F^n)$ be any function of $C^k$ smooth in its variables with $z\in\co, \eta_j\in\co^n$. The following differential system : $f=(f^1, ..., f^n)$
$$\frac{\partial^m }{\partial\bar z^m}f(z)=F(z, f(z), \frac{\partial }{\partial\bar z}f(z), ...,\frac{\partial^{m-1}}{\partial\bar z^{m-1}}f(z))$$
$$\frac{\partial^{i+j} }{\partial^i\partial\bar z^j}f(0)=a_{ij}\in \co^n,\mbox{ } i+j\leq m-1$$
has a solution of class $C^{m+k}$ near the origin.}

\bigskip
We note the solution is not unique unlike in ordinary differential equation of one real variable, however we will prove that the equation satisfies the unique continuation property in [P2].
\bigskip

More specific general results are given in the following subsections.
\subsection{Nonlinear systems-A general case}
First we introduce some simple notations. Let $D$ denote the closed disk $\{z\in\co\mid|z|\leq R\}$, $D'$ denote the closed disk $\{z\in\co\mid|z|\leq R'\}$ and $D'^n$ be the polycylinder of radius $R'$ in $\co^n$. Let $m$ be a positive integer. For each $k$: $1\leq k\leq m$, we denote $\mathcal{D}^k v$
as a vector in $\co^{2^k}$ with entries $\partial^\mu\bar\partial^\nu v$ where $\mu+\nu=k$ where $v$ is a complex valued function on $D$. Here we use the standard complex derivatives for partial derivatives as
$\partial=\frac{1}{2}(\frac{\partial}{\partial x}-i\frac{\partial}{\partial y}), \mbox{    }\bar\partial=\frac{1}{2}(\frac{\partial}{\partial x}+i\frac{\partial}{\partial y}),$
where $z=x+iy$.
Let $u$ a map from $D$ to $\D'^n$ with $u=(u^1,..., u^n)$. We denote $\mathcal{D}^k u$ as $(\mathcal{D}^k u^1,..., \mathcal{D}^k u^n)$, which is a vector in $\co^{2^kn}$.

Let nonnegative integers $\mu,\nu$ such that $\mu+\nu=m$, which will be fixed throughout the paper.
Let $\Omega$ be the domain
$$\Omega=D\times D'^n\times\co^{2n}\times\cdot\cdot\cdot\times\co^{2^{m-1}n}\times\co^{2^{m}n}$$
with coordinates as $(z, \eta_0, \eta_1, ..., \eta_{m-1}, \eta_m)$.  Also the norm of a vector is taken as the max norm for
any dimensional vector space. If $v\in \co^N$, then $|v|=\max_{1\leq i\leq n}|v^i|$. The notation $\partial_{\eta_m}a$ below is understood as a vector of partial derivatives of
$a$ with respect to variables in $\eta_m$, and so are similar notations. Also, it should be noted that a function $a(z, \eta_0, \eta_1, ...,\eta_m)$ is understood as
a function $a(z, \eta_0, \eta_1, ...,\eta_m,\bar z, \bar\eta_0, \bar\eta_1, ...,\bar\eta_m)$ in this paper. The following are the main results of this paper.

\begin{thm}\label{thm0}
   Let $a^i(z, \eta_0, \eta_1, ..., \eta_{m}):\Omega\to\co$ $(i=1,...,n)$ be functions of class $C^{k}(k\geq 2)$. Assume that
\begin{eqnarray}
    a(0)&=&\partial_{\eta_m} a(0)=\bar\partial_{\eta_m }a(0)=0,\\
    \partial_{\eta_m}\partial_{\eta_m} a(0)&=&\partial_{\eta_m}\bar\partial_{\eta_m} a(0)=\bar\partial_{\eta_m}\bar\partial_{\eta_m} a(0)=0.
\end{eqnarray}
  Then the following general partial differential system: $u=(u^1, ..., u^n)$
\begin{eqnarray}
  \partial^\mu\bar\partial^\nu u&=&a(z, u, \mathcal{D}^1 u, ...,\mathcal{D}^{m-1} u, \mathcal{D}^m u)
  \end{eqnarray}
  has solutions of class $C^{m+k-1+\alpha}$ of vanishing order $m$ at $0$ for sufficiently small values of $R$ for any $\alpha$ $(0<\alpha<1)$.

\end{thm}
Actually, more is true if the conditions (1) and (2) are replaced by a smallness conditions, but they are easy to verify.
\begin{thm}\label{thm0}
   Let $a(z, \eta_0, \eta_1, ..., \eta_{m})=(a^1,...,a^n)$ be a map from $\Omega$ to $\co^n$ of class $C^{k}(k\geq 2)$. There is a (small) constant $\delta<1$ depending only on $m$ and $\epsilon$ depending on the second derivatives of $a$ such that if
\begin{eqnarray}
    |a(0)|<\epsilon\nonumber\\
    |\partial_{\eta_m} a(0)|+|\bar\partial_{\eta_m }a(0)|<\delta\nonumber\\
    |\partial_{\eta_m}\partial_{\eta_m} a(0)|+|\partial_{\eta_m}\bar\partial_{\eta_m} a(0)|+|\bar\partial_{\eta_m}\bar\partial_{\eta_m} a(0)|&<&\delta,
\end{eqnarray}
  then the system (3)
  has solutions of class $C^{m+k-1+\alpha}$ of vanishing order $m$ at $0$ for sufficiently small values of $R$.
  \end{thm}
We point out that the vanishing conditions (1) and  (2) can't be dropped due to the counterexample of Mizohata (see below 1.6).
When the right hand side of system (3) is independent of $m$th derivatives, the solution can be chosen freely for orders up to $m-1$ at the origin. The conditions (1) and (2) become unnecessary.
We point out that this result is not a consequence of the above theorem but the method of its proof.

\begin{thm}\label{thm0}
   Let $a(z, \eta_0, \eta_1, ..., \eta_{m-1})=(a^1,...,a^n)$ be a map from $\Omega$ to $\co^n$ of class $C^{k}(k\geq 2)$ that is independent of $\eta_m$. Let $p^i(z,\bar z)$ be any polynomial of degree at most $m-1$ such that $p^i(0,0)\in \mathrm{Int}(D')$ $(i=1,...,n)$.
  Then the following partial differential system: $u=(u^1, ..., u^n)$
\begin{eqnarray}
  \partial^\mu\bar\partial^\nu u&=&a(z, u, \mathcal{D}^1 u, ...,\mathcal{D}^{m-1} u)
  \end{eqnarray}
  has non-constant solutions of class $C^{m+k+\alpha}$ for sufficiently small values of $R$ for any $\alpha$ $(0<\alpha<1)$. The solutions can be chosen near the origin
  $$u^i=p^i(z,\bar z)+O(|z|^m).$$

\end{thm}
When the function $a$ is independent of $z$, or the system (3) is so-called autonomous, we can provide global solutions of vanishing order $m$ at the origin. Let $K$ be a polycylinder in $\Omega$ containing the origin. We note $K$ is a compact set.
\begin{thm}\label{thm0}
   Let $a(\eta_0, \eta_1, ..., \eta_{m})=(a^1,...,a^n)$ be a map from $\Omega$ to $\co^n$ of class $C^{k}(k\geq 2)$ that is independent of $z$; let $\alpha$
   be given $(0<\alpha<1)$.
   There is a (small) constant $\tau$ depending only on $m, \alpha, R$ and there is an $\epsilon$ dependent on $\tau$ and the maximum norm of second derivatives
   of $a$ over $K$ such that if
\begin{eqnarray}
   |a(0)|<\epsilon,\mbox{  and  } |\nabla a(0)|<\tau
\end{eqnarray}
  then, for any homogenous polynomial map $P_m(z,\bar z)$of degree $m$ without term $z^\mu\bar z^\nu$ for which all coefficients have absolute value less that
  $\epsilon$, the system
  \begin{eqnarray}
  \partial^\mu\bar\partial^\nu u&=&a(u, \mathcal{D}^1 u, ...,\mathcal{D}^{m} u)
  \end{eqnarray}
  has solutions of class $C^{m+k-1+\alpha}$ in the whole $D$ such that at the origin: $\partial^i\bar\partial^ju(0)=\partial^i\bar\partial^jP_m(0)$ for all $i+j=m$
  and $i\not=\mu, j\not=\nu$ and $\partial^i\bar\partial^ju(0)=0$ for $i+j\leq m-1$.
\end{thm}
The result can be actually proved for small jets of order less than $m$. That is, we can prove the following more general existence theorem.
\begin{thm}\label{thm0}
   Let $a(\eta_0, \eta_1, ..., \eta_{m})=(a^1,...,a^n)$ be a map from $\Omega$ to $\co^n$ of class $C^{k}(k\geq 2)$ that is independent of $z$; let $\alpha$
   be given $(0<\alpha<1)$.
   There is a (small) constant $\tau$ depending only on $m, \alpha, R$ and there is an $\epsilon$ dependent on $\tau$ and the maximum norm of second derivatives
   of $a$ over $K$ such that if
\begin{eqnarray}
   |a(0)|<\epsilon,\mbox{  and  } |\nabla a(0)|<\tau
\end{eqnarray}
  then, for any polynomial map $p(z)$ of degree less than or equal to $m-1$ whose coefficients are less than $\epsilon$ in absolute value, and any homogenous polynomial map $P_m(z,\bar z)$of degree $m$ without term $z^\mu\bar z^\nu$ for which all coefficients have absolute value less than
  $\epsilon$, the system (7)
  has solutions of class $C^{m+k+\alpha}$ in the whole $D$ such that at the origin: $\partial^i\bar\partial^ju(0)=\partial^i\bar\partial^jP_m(0)$ for all $i+j=m$
  and $i\not=\mu, j\not=\nu$ and $\partial^i\bar\partial^ju(0)=\partial^i\bar\partial^jp(0)$ for $i+j\leq m-1$.
\end{thm}
\subsection{Real systems}
In the previous sections, solutions found are generally complex valued (vector). However when $\mu=\nu$, we can find real-valued solutions that could be useful
for studying Riemannian geometry like harmonic maps and biharmonic maps and to define Kobayashi metric on tangent bundle.
Here we will consider everything in real variables. First we consider that $D$ and $D'$ are defined in $\re^2$, then define
$$\Omega_{\re}=D\times D'^n\times\re^{2n}\times\re^{2^2n}\times\cdot\cdot\cdot\times\re^{2^{2m-1}n}.$$
\begin{thm}
Let $A^i(x,y;\xi_0,\xi_1,...,\xi_{2m-1}):\Omega_{\re}\to \re$ be functions of class $C^{k} (k\geq 2)$; let $u^i(x,y):D\to\re$
be unknown functions for $i=1,...,n$. Let $p^i(x,y)$ be real polynomial of degree at most $2m-1$ such that $p^i(0)\in \mathrm{Int}(D')$.
Then the following $m$-Laplace system equation $u=(u^1,...,u^n)$:
$$\Delta^m u^i(x,y)=A^i(x,y;u,\nabla u, \nabla^2 u,...,\nabla^{2m-1}u), \mbox{                  }i=1,...,n$$
has solutions of class $C^{2m+k+\alpha}$ for sufficiently small values of $R$ for a given $\alpha$ $(0<\alpha<1)$ so that near the origin
$$u^i(x,y)=p^i(x,y)+O(\sqrt{x^2+y^2})^{2m}.$$
\end{thm}
This theorem when applied with $m=1$ gives a proof of Theorem B.

The following is about autonomous systems and global solutions are found.
\begin{thm}
Let $A^i(\xi_0,\xi_1,...,\xi_{2m-1},\xi_{2m}):\Omega_{\re}\times\re^{2^{2m}n}\to \re$ be functions of class $C^k (k\geq 2)$ that is independent of $x,y$; There is constant $\tau>0$ depending on $R,m,\alpha$ such that if $A^i$ satisfies
$$A^i(0)=0, |\nabla A^i(0)|<\tau,\mbox{ } i=1,...,n$$
then the following $m$-Laplace system equation $u=(u^1,...,u^n)$:
$$\Delta^m u^i(x,y)=A^i(u,\nabla u, \nabla^2 u,...,\nabla^{2m-1}u,\nabla^{2m}u), \mbox{   }i=1,...,n$$
has solutions of class $C^{2m+k+\alpha}$ in the whole $D$ for a given $\alpha$ $(0<\alpha<1)$ with vanishing order $2m$ at the origin.
\end{thm}

\subsection{Quasi-linear systems}
In the paper sequential to this paper [P1], we will reduce the regularity of $a\in C^k (k\geq 2)$ when $a$ is linear in the variable $\eta_m$ to
the case $a\in C^{k,\alpha}$ where $k\geq 0, 0<\alpha<1.$ Let
$$\Omega'=D\times D'^n\times\co^{2n}\times...\times\co^{2^{m-1}n}.$$
\begin{thm}
 Let $a_{k,l}(z,\eta_1,...,\eta_{m-1}),b_{k,l}(z,\eta_1,...,\eta_{m-1}),c(z,\eta_1,...,\eta_{m-1})$ be $n\times n$ matrices and $b(z,\eta_1,...,\eta_{m-1})$ be
an $n\times 1$ matrix of class $C^{k,\alpha}(k\geq 0,0<\alpha<1)$ defined on $\Omega'$; $a_{k,l}(0)=b_{k,l}(0)=c(0)=0,b(0)=0$; let $u=(u^1,...,u^n)$ be unknown functions on  a disk $D$
in $\co$,
with radius $R$. Then for $\partial^i\bar\partial^j u(0)=0$ for $i+j\leq m-1$, and with given (vector) values of $\partial^i\bar\partial^j u(0)$ with $i+j=m, i\not=\mu
,j\not=\nu$, the differential system
$$\partial^\mu\bar\partial^\nu u=b(z, u, \mathcal{D}^1 u, ...,\mathcal{D}^{m-1} u)+\sum_{k+l=m}a_{k,l}(z, u, \mathcal{D}^1 u, ...,\mathcal{D}^{m-1} u)\partial^k\bar\partial^l u+$$
$$\sum_{k+l=m}b_{k,l}(z, u, \mathcal{D}^1 u, ...,\mathcal{D}^{m-1} u)\overline{\partial^k\bar\partial^l u}+c(z, u, \mathcal{D}^1 u, ...,\mathcal{D}^{m-1}u)\overline{\partial^\mu\bar\partial^\nu u}$$
has a solution of class $C^{m+k+\alpha}$ for sufficiently small values of $R$.
\end{thm}
The proof of the result when $k=0$ requires the application of Schauder's Fixed Point Theorem. For the case $k\geq 1$, it is similar to that of Theorem 1.2.

When $a, b, c, d$ are independent of $z$, global solutions can be found as in Theorem 1.4, 1.5.
For a system of first order, we state the theorem as a corollary, because of its importance with $J$-holomorphic curves.
\begin{cor} Let $a(z,\eta), b(z,\eta), c(z,\eta)$ be $n\times n$ matrices and $d(z,\eta)$ be an $n\times 1$ matrix
defined on $D\times D'$ of class $C^{k,\alpha} (k\geq 0,0<\alpha<1)$; $a(0)=b(0)=c(0)=d(0)=0$;let $u=(u^1,...,u^n)$ be unknown functions on  a disk $D$
in $\co$, with radius $R$. Then for $u(0)=0$, and with given (vector) values of $\partial u(0)$, the differential system of first order
$$\bar\partial u=a(z,u)\overline {\partial u}+b(z,u)\partial u++c(z, u)\partial \bar u+d(z,u),$$
has a solution of class $C^{m+k+\alpha}$ for sufficiently small values of $R$.
\end{cor}
In above, if $b=c=d=0$, $a(z,u)=a(u)$, independent of $z$ and $a(0)=0$, then the equation reduces to one which defines J-holomorphic curves on an almost complex manifold. This is the theorem of
Nijenhuis and Woolf. if $b=c=0$ only, $a(z,u)=a(u)$ and $a(0)=0.$
In terms of partial derivatives $\partial_x, \partial_y$, Theorem 1.9 can be converted to a system of order $m$ with some vanishing conditions.
\begin{thm}
Let $a_{kl}^j(z,\eta_0,...,\eta_{m-1}), f^i(z,\eta_0,...,\eta_{m-1}), k+l=m, j=1,...,n$ be $n2^m +n$ functions defined on $\Omega'$ of class $C^{k,\alpha}(k\geq 0, 0<\alpha<1)$.
Let $u=(u^1,...,u^n)$ be unknown functions and $X=(z, u, \nabla u, ..., \nabla^{m-1} u)$; let
$$C_p^j(z,\eta_0,...,\eta_{m-1})=\sum_{k+l=m}a_{kl}^j(z,\eta_0,...,\eta_{m-1}) i^l\sum_{q=\max\{0,p-k\}}^{\min\{l,p\}}{k\choose p-q}{l\choose q}(-1)^{l-q}$$
for $j=1,...,n; p=0,1,...,m$. If there is $p_0\in\{0,1,...,m\}$ such that
$$C_{p_0}^j(0)\not=0; j=1,...,n, $$
$$C_p^j(0)=0; p\not=p_0, j=1,...,n,$$
then the differential system of order $m$
$$\sum_{k+l=m}a_{kl}^j(X)\partial_x^k\partial^l_y u^j(x,y)=f^j(X); j=1,...,n$$
has a solution of class $C^{m+k+\alpha}$ that is of vanishing order $m$ for sufficiently small values of $R$.
\end{thm}
This is a consequence of Theorem 1.8 and a lemma in Appendix. All proofs of the above results will be given in [P1].

\subsection{Holomorphic systems}
Let
$$\Omega_4=D\times D'\times\co^n\times\cdot\cdot\cdot\times\co^n$$
We denote $f^{(k)}(z)=\partial^k f(z)$. Our method allows to prove the following local existence of holomorphic solutions.
\begin{thm}
Let $H(z, \eta_0,...,\eta_{n-1}):\Omega_4\to\co^n$ be a mapping of class $C^k(k\geq 2)$ that is holomorphic in $\mathrm{Int}(\Omega_4)$;  Then the following holomorphic differential system
$$f^{(m)}(z)=H(z,f(z), f'(z),....,f^{(m-1)}(z))$$
with initial values
\begin{eqnarray*}
f(0)&=&a_0 \in \mathrm{Int}(D)\\
f^{(i)}(0)&=&a_i,\mbox{ } i=1,...,m-1
\end{eqnarray*}
has a unique holomorphic solution for sufficiently small values of $R$.
\end{thm}
This theorem is slightly more general than the classical fundamental theorem of holomorphic ordinary differential system(see [IY]).
However, the following provides global solutions if $H$ is independent of $z$ and seems new.
\begin{thm}
Let $H(\eta_0,...,\eta_{n-1}):\Omega_4\to\co^n$ be a mapping of class $C^k(k\geq 2)$ that is holomorphic in $\mathrm{Int}(\Omega_4)$ and is independent of
$z$. Assume that
$$H(0)=0,\mbox{ } \partial H(0)=0.$$
Then the following holomorphic differential system
$$f^{(m)}(z)=H(f(z), f'(z),....,f^{(m-1)}(z))$$
with initial values
\begin{eqnarray*}
f(0)&=&a_0 \in \mathrm{Int}(D)\\
 f^{(i)}(0)&=&a_i,\mbox{ } i=1,...,m-1
\end{eqnarray*}
has a unique solution in the whole $D$ of class $C^{m+k}$, holomorphic in $\mathrm{Int}(D)$.
\end{thm}
\subsection{Dependence on parameters}
If $a$ as in previous theorems depends on some additional parameters in a $C^{k'}$ fashion $(0\leq k'\leq k)$, we will show the existence of parameterized  solutions with the smoothness properties in parameters. Proofs are given [P].
\subsection{Examples of no solutions-Mizohata equation}
After a famous example of Lewy [L] in $\re^3$, Mizohata [M] considered, in $\re^2$, the following equation
$$\frac{\partial u}{\partial x}+i x\frac{\partial u}{\partial y}=F(x,y)\label{eq0}.$$
It was proved in [M] that there is a smooth function $F$ for which the above equation has no solution near the origin. Converting the equation to complex one, one has the following equation
\begin{eqnarray}
\bar\partial u =\frac{1}{1+\rp z}F(z,\bar z)-\frac{1-\rp z}{1+\rp z}\partial u.
\end{eqnarray}
According to the notation as Theorem 1.1, we have
$$a(z, \eta_0, \eta)=\frac{1}{1+\rp z}F(z,\bar z)-\frac{1-\rp z}{1+\rp z}\eta,$$
whence
$$\partial_\eta(0)=-1, \bar\partial_\eta a(0)=0.$$
Therefore the condition (1) is not met for Theorem 1.1. Taking $\partial^\mu\bar\partial^{\nu-1}$ on both side, we have
$$\partial^\mu\bar\partial^{\nu}u=-\frac{1-\rp z}{1+\rp z}\partial^{\mu+1}\bar\partial^{\nu-1}u+\cdot\cdot\cdot+\partial^\mu\bar\partial^{\nu-1}\left\{\frac{1}{1+\rp z}F(z,\bar z)\right\}.$$
This is a differential equation of any order which has no solutions locally at the origin.

The organization of this paper is as follows. In Section 2, and 3, we prove various results for setting up later proofs of using fixed point theorem on a Banach space, and the other sections provide proofs of theorems. We rely on some of classical results for Green operator from [NW], but otherwise the paper is
self-contained. Undoubtedly, the work in this paper is largely influenced by the fundamental works of
Newlander-Nirenberg [NN] on almost complex structures and Nijenhuis-Woolf [NW] on local existence of $J$-holomorphic curves.

\section{Function spaces and their norms}

Let $D$ denote the closed disk $\{z\in\co\mid|z|\leq R\}$ and $C$ its boundary $\{z\in\co \mid |z|=\R\}$. Unless otherwise stated, all functions considered will be complex valued and integrable, with domain $D$. We will consider several classes of functions:

\subsection{H\" older space}

$C^{\alpha}(D)$ is the set of all functions $f$ on $D$ for which

$$ H_{\alpha}[f]=\sup\left\{{\frac{|f(z)-f(z')|}{|z-z'|^{\alpha}} \bigg| z,z' \in \D} \right\}$$
is finite.

$C^{k}(D)$ is the set of all function $f$ on $D$ whose $k^{\textup{th}}$ order partial derivatives exist and are continuous, $k$ an integer, $k \geq 0$.
$C^{k+\alpha}(D)$ is the set of all functions $f$ on $D$ whose $k^{\textup{th}}$ order partial derivatives exist and belong to $C^{\alpha}(D)$.

The symbol $|f|$ or $|f|_\D$ denotes $\textup{sup}_{z\in D}|f(z)|$.
For $f\in C^\alpha(D)$ we define
$$\|f\|=|f|+(2R)^\alpha H_\alpha[f].$$
The set of $n$-tuples $f=(f_1,..., f_n)$ of functions (vector functions or maps) of $C^\alpha (D)$ is denoted $[C^\alpha (D)]^n$, and $H_\alpha[f]$ is defined as the maximum of $H_\alpha[f_i](i=1,..,n)$. In a similar fashion we define $|f|_A=\sup_{z\in A}|f(z)|$ for functions and vector functions, and write $|f|$ when the domain is understood. Finally, in this paper throughout, the norm of $\co^N$ is taken as $|v|=\max|v_j|$.

The following lemma is well-known; for a proof see [NW].
\begin{lem}
The function $\|\cdot\cdot\cdot\|$ defined on $C^{\alpha}(D)$ is a norm, with respect to which $C^{\alpha}(D)$ is a Banach algebra: $\|fg\|\leq \|f\|\|g\|$.
\end{lem}
The following simple lemma is to be used multiple times throughout the paper.
\begin{lem} If $f\in C^{k+\alpha}(D)$, then
$$|f(z')-\sum_{l=0}^k\frac{1}{l!}\sum_{i+j=l}\partial^i\bar\partial^j f(z)(z'-z)^i{(\bar z'-\bar z)}^j|\leq \frac{1}{k!}\bigg\{\sum_{i+j=k}H_\alpha[\partial^i\bar\partial^j f]\bigg\}|z'-z|^{k+\alpha}.$$
\end{lem}
\begin{proof}
Expanding at $z$, we have the formula
\begin{eqnarray*}
&&f(z')-\sum_{l=0}^{k-1}\frac{1}{l!}\sum_{i+j=l}\partial^i\bar\partial^j f(z)(z'-z)^i{(\bar z'-\bar z)}^j\\
&=&\int_0^1\int_0^{t_{k-1}}\cdot\cdot\cdot\int_0^{t_1}\left\{\frac{d^k}{dt^k}f_N(t z'+(1-t)z)\right\}dtdt_1\cdot\cdot\cdot dt_{k-1}\nonumber\\
&=&\int_0^1\int_0^{t_{k-1}}\cdot\cdot\cdot\int_0^{t_1}\left\{\sum_{i+j=k}\partial^i\bar\partial^j f(tz'+(1-t)z)(z'-z)^i{(\bar z'-\bar z)}^j\right\}dtdt_1\cdot\cdot\cdot dt_{k-1}.\nonumber
\end{eqnarray*}
Hence, we have, by subtracting kth term,
$$f(z')-\sum_{l=0}^k\frac{1}{l!}\sum_{i+j=l}\partial^i\bar\partial^j f(z)(z'-z)^i{(\bar z'-\bar z)}^j$$
$$=\int_0^1\int_0^{t_{k-1}}\cdot\cdot\cdot\int_0^{t_1}\left\{\sum_{i+j=k}\{\partial^i\bar\partial^j f(tz'+(1-t)z)-\partial^i\bar\partial^j f(z)\}(z'-z)^i{(\bar z'-\bar z)}^j\right\}dtdt_1\cdot\cdot\cdot dt_{k-1}$$
Thus we have,
$$|f(z')-\sum_{l=0}^k\frac{1}{l!}\sum_{i+j=l}\partial^i\bar\partial^j f(z)(z'-z)^i{(\bar z'-\bar z)}^j|$$
$$\leq \frac{1}{k!}\sum_{i+j=k}H_\alpha[\partial^i\bar\partial^j f]|z'-z|^{k+\alpha}.$$
This completes the proof.
\end{proof}
\subsection{Function spaces with high order vanishing at the origin}
Our idea of solving differential equations of order $m$ is to look for solutions that vanish up to $m-1$ order at the origin; this way the norm estimate of function space to be considered later is made possible using only $m$th order derivatives. This is rather different from classical norms
used for higher order derivatives in partial differential equations. This idea could be applied to higher dimensional differential equations.
We denote for $k\geq 1$, $C_0^{k+\alpha}(D)$ the set of all functions in $C^{k+\alpha}(D)$ whose derivatives vanish up to order $k-1$ at the origin. Specifically
$$C_0^{k+\alpha}(D)=\{f\in  C^{k+\alpha}(D)\big | \partial^i\bar \partial^j f(0)=0, i+j\leq k-1\}.$$
One has the following obvious nesting
$$C_0^{m+\alpha}(D)\subset C_0^{m-1+\alpha}(D)\subset\cdot\cdot\cdot\subset C_0^{1+\alpha}(D)\subset C^{\alpha}(D).$$
We now define functions $\|\cdot\cdot\cdot\|^{(k)}$ on $C^{k+\alpha}(D)$ inductively. On $C^{1+\alpha}(D)$ we define, following [NW],
$$\|f\|^{(1)}=\max\{\|\partial f\|, \|\bar\partial f\|\}.$$
For $k\geq 2$, we define
$$\|f\|^{(k)}=\max\{\|\partial f\|^{(k-1)}, \|\bar\partial f\|^{(k-1)}\}.$$
Obviously, we have the definition in terms of $\|\cdot\cdot\cdot\|$:
$$\|f\|^{(k)}=\max_{i+j=k}\{\|\partial^i\bar\partial^j f\|\}.$$
We point out that the function $\|\cdot\cdot\cdot\|^{(k)}$ on $C^{k+\alpha}(D)$ is not a norm since $\|f\|^{(k)}=0$ if and only if $f$ is
a polynomial of degree $k-1$. However it becomes norm when restricted to subspaces $C_0^{k+\alpha}(D)$, which is to be proved below and is one of important facts used in
this paper. First we obtain some useful estimates.
\begin{lem}
If $f\in C_0^{k+\alpha}(D)$, then
$$\|f\|\leq \frac{6^k}{k!}R^k\|f\|^{(k)}.$$
\end{lem}
\begin{proof}
Let $f\in C_0^{k+\alpha}(D)$, then
\begin{eqnarray*}
f(z)&=&\int_0^1\int_0^{t_{k-1}}\cdot\cdot\cdot\int_0^{t_1}\left\{\frac{d^k}{dt^k}f(t z)\right\}dtdt_1\cdot\cdot\cdot dt_{k-1}\nonumber\\
 &=&\int_0^1\int_0^{t_{k-1}}\cdot\cdot\cdot\int_0^{t_1}\left\{\sum_{i+j=k}\partial^i\bar\partial^j f(t z)z^i\bar z^j \right\}dtdt_1\cdot\cdot\cdot dt_{k-1}\nonumber\\
&=&\sum_{i+j=k}\left\{\int_0^1\int_0^{t_{k-1}}\cdot\cdot\cdot\int_0^{t_1}\partial^i\bar\partial^j f(t z) dtdt_1\cdot\cdot\cdot dt_{k-1}\right\}z^i\bar z^j\nonumber
\end{eqnarray*}
Applying norm inequality, we obtain
\begin{eqnarray*}
\|f\|&\leq & \sum_{i+j=k}\frac{1}{k!}\|\partial^i\bar\partial^j f\|\|z^i\bar z^j\|\nonumber\\
&\leq &\sum_{i+j=k}\frac{1}{k!}\|\partial^i\bar\partial^j f\|\|z\|^k\leq\frac{2^k}{k!}(3R)^k\|f\|^{(k)}\nonumber,
\end{eqnarray*}
where we have used $\|z\|=3R$, which is easily verified.
\end{proof}
\begin{lem}
If $f\in C_0^{m+\alpha}(D)$, then, for $i+j=l\leq m$,
$$\|\partial ^i\bar\partial^j f\|\leq \frac{6^{m-l}}{(m-l)!}R^{m-l}\|f\|^{(m)}.$$
\end{lem}
\begin{proof}
Let  $f\in C_0^{m+\alpha}(D)$. If $i+j=l$, then $\partial^i\bar\partial^j f\in C_0^{m-l+\alpha}(D)$. By Lemma 2.3, we have
$$\|\partial^i\bar\partial^j f\|\leq \frac{6^{m-l}}{(m-l)!}R^{m-l}\|\partial^i\bar\partial^j f\|^{(m-l)}\leq \frac{6^{m-l}}{(m-l)!}R^{m-l}\| f\|^{(m)}.$$
\end{proof}
An immediate corollary is the following:
\begin{lem}
If $f\in C_0^{m+\alpha}(D)$, then, for $l\leq m$,
$$\|f\|^{(l)}\leq \frac{6^{m-l}}{(m-l)!}R^{m-l}\|f\|^{(m)}.$$
\end{lem}

In order to verify that $C_0^{k+\alpha}(D)$ is a Banach space with norm $\|\cdot\cdot\cdot\|^{(k)}$, we need the following from [NW](7.1a).
\begin{lem}
Let $\{f_N\}_{N=1}^\infty$ be a sequence in $C^\alpha(D)$, with $\|f_N\|\leq M$; let $\{f_N\}$ converges at each point of a dense subset $A$ of $D$. Then $\{f_N\}$ converges to a function $f$ on $D$, in the norm $|\cdot\cdot\cdot|$; and $f\in C^\alpha(D)$, $\|f\|\leq M$.
\end{lem}
\begin{lem}
Let $\{f_N\}$ be a sequence in $C_0^{k+\alpha}(D)$, with $\|f_N\|^{(k)}\leq M$, and if $\{\partial^i\bar\partial^j f_N\}$, for all $i,j, i+j=k$, are Cauchy sequences in the norm $|\cdot\cdot\cdot|$, then there is a function $f\in C_0^{k+\alpha}(D)$ such that $|\partial^i\bar\partial^j f_N-\partial^i\bar\partial^j f|\to 0$ as $N\to\infty$ for all $i,j, 0\leq i+j\leq k$, and with $\|f\|^{(k)}\leq M$.
\end{lem}
\begin{proof}
For $i, j, i+j=k$, consider $g_N^{i,j}=\partial^i\bar\partial^j f_N$, then $g_N^{i,j}\in C^\alpha(D)$ and $\|g_N^{i,j}\|\leq M$. Applying to Lemma 2.6, we
have functions $g^{i,j}\in C^\alpha(D)$ such that $|g_N^{i,j}-g^{i,j}|\to 0$ as $ N \to \infty$, with $\|g^{i,j}\|\leq M $.
Define $f$ by
$$f(z)=\int_0^1\int_0^{t_{k-1}}\cdot\cdot\cdot\int_0^{t_1}\left\{\sum_{i+j=k}g^{i,j} (t z)z^i\bar z^j \right\} dtdt_1\cdot\cdot\cdot dt_{k-1}.$$
We have
$$f_N(z)=\sum_{i+j=k}\left\{\int_0^1\int_0^{t_{k-1}}\cdot\cdot\cdot\int_0^{t_1}\partial^i\bar\partial^j f_N(t z) dtdt_1\cdot\cdot\cdot dt_{k-1}\right\}z^i\bar z^j$$
Therefore
$$f_N(z)-f(z)=\sum_{i+j=k}\left\{\int_0^1\int_0^{t_{k-1}}\cdot\cdot\cdot\int_0^{t_1}\left\{\partial^i\bar\partial^j f_N(t z)-g^{i,j}(t z)\right\} dtdt_1\cdot\cdot\cdot dt_{k-1}\right\}z^i\bar z^j,$$
whence
$$|f_N-f|\leq \frac{R^k}{k!}\sum_{i+j=k}|\partial^i\bar\partial^j f_N-g^{i,j}|,$$
which goes to $0$ as $N\to\infty$, implying $f_N\to f$ in the norm $|\cdot\cdot\cdot|$.

For $i+j=l\leq k-1$ we want to prove that $\{\partial^i\bar\partial^j f_N(z)\}$ are Cauchy sequences. Indeed,
Since $f_N$ vanishes up to $ k-1$ order at the origin, then for $i+j=l\leq k-1$, $\partial ^i\bar\partial^j f_N$ vanishes to
$k-1-l$ order at the origin. Thus, we have the formula, applying (7)
\begin{eqnarray}
\partial^i\bar\partial^j f_N(z)&=&\int_0^1\cdot\cdot\cdot\int_0^{t_{m-1-l}}\frac{d^{m-l}}{dt^{m-l}}\partial^i\bar\partial^j f_N(tz)dt\cdot\cdot\cdot dt_{m-1-l}\nonumber\\
&=&\int_0^1\cdot\cdot\cdot\int_0^{t_{m-1-l}}\sum_{p+q=m-l}\partial^p\bar\partial^q\partial^i\bar\partial^j f_N(tz)z^p\bar z^qdt\cdot\cdot\cdot dt_{m-1-l}\nonumber\\
&=&\sum_{p+q=m-l}\int_0^1\cdot\cdot\cdot\int_0^{t_{k-1-l}}\partial^p\bar\partial^q\partial^i\bar\partial^j f_N(tz)dt\cdot\cdot\cdot dt_{k-1-l}z^p\bar z^q,
\end{eqnarray}
Then
\begin{eqnarray}
&&\partial ^i\bar\partial^j f_N(z)-\partial ^i\bar\partial^j f_{N'}(z)\nonumber\\
&=&\sum_{p+q=k-l}\int_0^1\cdot\cdot\cdot\int_0^{t_{m-1-l}}\partial^p\bar\partial^q\partial^i\bar\partial^j \{f_N(tz)-f_{N'}(tz)\}dt\cdot\cdot\cdot dt_{k-1-l}z^p\bar z^q,
\end{eqnarray}
Then
\begin{eqnarray}
|\partial ^i\bar\partial^j f_N(z)-\partial ^i\bar\partial^j f_{N'}(z)|\leq\frac{R^{m-l}}{(m-l)!}\sum_{i+j=k}|\partial ^i\bar\partial^j f_N-\partial ^i\bar\partial^j f_{N'}|
\end{eqnarray}
which proves that $\{\partial^i\bar\partial^j f_N(z)\} (i+j\leq k-1)$ are Cauchy sequences since $\{\partial^i\bar\partial^j f_N(z)\} (i+j=k)$ are.
We assume that for $i+j=l\leq k-1$, $\partial ^i\bar\partial^j f_N(z)$ converges to $g^{i,j}$ in norm $|\cdot\cdot\cdot|$.
Thus, applying Lemma 2.2,  we have
\begin{eqnarray*}
&&\bigg|f_N(z')-\sum_{l=0}^k\frac{1}{l!}\sum_{i+j=l}\partial^i\bar\partial^j f_N(z)(z'-z)^i{(\bar z'-\bar z)}^j\bigg|\nonumber\\
&\leq& \frac{1}{k!}\sum_{i+j=k}H_\alpha[\partial^i\bar\partial^j f_N]|z'-z|^{k+\alpha}\nonumber\\
&\leq& \frac{2^k}{k!}(2R)^{-\alpha}\|f_N\|^{(k)}|z'-z|^{k+\alpha}
\leq \frac{2^k}{k!}(2R)^{-\alpha}M|z'-z|^{k+\alpha},\nonumber
\end{eqnarray*}
which is independent of $N$, letting $N\to \infty$
we have
$$|f(z')-f(z)-\sum_{l=1}^k\frac{1}{l!}\sum_{i+j=l}g^{i,j}(z'-z)^i{(\bar z'-\bar z)}^j|\leq \frac{2^k}{k!}(2R)^{-\alpha}M|z'-z|^{k+\alpha},$$
which implies, by definition of differentiability, $g^{i,j}=\partial^i\bar\partial^j f$. This implies $\|f\|^{(k)}\leq M$  by taking limit
from $\|f_N\|^{(k)}\leq M$. The convergence for $i+j\leq k-1$, follows from that of $i+j=k$ by (12).
\end{proof}
\begin{lem}
The function space $C_0^{k+\alpha}(D)$ equipped with the function $\|\cdot\cdot\cdot\|^{(k)}$ is a Banach space.
\end{lem}
\begin{proof}
Let $\{f_N\}$ be a Cauchy sequence in $\|\cdot\cdot\cdot\|^{(k)}$. Then for any $\epsilon>0$, there is $N_0$ so that if $N, N'>N_0$ it holds
\begin{eqnarray}
\|f_N-f_{N'}\|^{(k)}<\epsilon,
\end{eqnarray}
which implies
$$|\|f_N\|^{(k)}-\|f_{N'}\|^{(k)}|<\epsilon,$$
which implies $\{\|f_N\|^{(k)}\}$ are Cauchy sequence and therefore bounded by say $M$. Also by definition of $\|\cdot\cdot\cdot\|^{(k)}$, (?)implies $\partial^i\bar\partial^j f_N$ for
$i+j=k$ are Cauchy sequence in $|\cdot\cdot\cdot|$. By Lemma 2.7, there is a function $f\in C^{k+\alpha}_0(D)$ such that
$|\partial^i\bar\partial^j f_N-\partial^i\bar\partial^j f|\to 0$ for $i+j=k$ as $N\to\infty$, and $\|f\|^{(k)}\leq M$.

Now the sequence $\{f_N-f_{N'}\}_{N'=N+1}^\infty$ is bounded in $\|\cdot\cdot\cdot\|^{(k)}$ by $\epsilon$, and converges to $f_N-f$, with
$$\partial^i\bar\partial^j (f_N-f_{N'})\to \partial^i\bar\partial^j (f_N-f)$$
in $|\cdot\cdot\cdot|$ as $N'\to\infty$, $i+j=k$. By Lemma 2.7 again, it holds
$$\|f_N-f\|^{(k)}<\epsilon,$$
which implies that $f_N\to f$ in $\|\cdot\cdot\cdot\|^{(k)}$. The proof is complete.
\end{proof}

\section{Green operator and high order derivative formula}
\subsection{Basic definitions and properties}
The operators are defined for integrable functions on $D$ as follows:
\begin{eqnarray}
 Tf(z)&=&\frac{-1}{2\pi i}\int_D\frac{f(\zeta)d\bar{\zeta}\wedge d\zeta}{\zeta-z},\nonumber\\
Sf(z)&=&\frac{1}{2\pi i}\int_C\frac{f(\zeta)d\zeta}{\zeta-z},\nonumber\\
 S_bf(z)&=&\frac{1}{2\pi i}\int_C\frac{f(\zeta)d\bar\zeta}{\zeta-z},\nonumber\\
^2Tf(z)&=&\frac{-1}{2\pi i}\int_D\frac{f(\zeta)-f(z)}{(\zeta-z)^2}d\bar\zeta\wedge d\zeta.\nonumber
\end{eqnarray}
Using polar coordinates, we see that $Tf$ is defined for continuous $f$, and $^2Tf$ for $C^\alpha(D)$. The operator $S$ is the familiar Cauchy integral.
We also define related operators: $\overline T, \overline S,\overline S_b$, and $^2\overline T$ as follows: $\overline T(f)=\overline {T(\bar f)}$, $\overline S(f)=\overline {S(\bar f)}$, $\overline S_b(f)=\overline {S_b(\bar f)}$ and $^2\overline T(f)=\overline {^2T(\bar f)}$ . Specifically, they are given by
\begin{eqnarray}
 \overline T f(z)&=&\frac{-1}{2\pi i}\int_D\frac{f(\zeta)d\bar{\zeta}\wedge d\zeta}{\bar\zeta-\bar z},\nonumber\\
 \overline Sf(z)&=&\frac{-1}{2\pi i}\int_C\frac{f(\zeta)d\bar\zeta}{\bar\zeta-\bar z},\nonumber\\
  \overline S_bf(z)&=&\frac{-1}{2\pi i}\int_C\frac{f(\zeta)d\zeta}{\bar\zeta-\bar z},\nonumber\\
^2\overline T f(z)&=&\frac{-1}{2\pi i}\int_D\frac{f(\zeta)-f(z)}{(\bar \zeta-\bar z)^2}d\bar\zeta\wedge d\zeta.\nonumber
\end{eqnarray}
Importantly, we will use operators of $T^\mu\overline{T}^\nu$, which is a composition of $T,\overline T$. Here we assume $T^0=\mathrm{Id},\overline T^0=\mathrm{Id}$ and etc.
We note here that $S_b$ in this paper is the same as $\overline S$ in [NW]. The following estimate holds:
\begin{eqnarray}
 |Tf|\leq 4R|f|.
\end{eqnarray}
More generally, if $\triangle$ is a bounded domain, then $T_\triangle f$ is defined for continuous $f$ on $\triangle$ by
$$T_\triangle f(z)=\frac{-1}{2\pi i}\int_\triangle\frac{f(\zeta)d\bar{\zeta}\wedge d\zeta}{\zeta-z}, \mbox{ }S_\triangle f(z)=\frac{1}{2\pi i}\int_{\partial\triangle}\frac{f(\zeta) d\zeta}{\zeta-z} .$$
We have
$$ |T_\triangle f|\leq 2 \mathrm{diam}(\triangle)|f|_\triangle.$$
The fundamental property between operators $T,S$ is the following [NW](6.1a).
\begin{lem}
If $f\in C^1(D)$, then
$$ T\bar\partial f=f-Sf     \mbox{         on         }\mbox{} \mathrm{Int}(D),$$
$$ \overline T\partial f=f-\overline Sf     \mbox{         on         }\mbox{} \mathrm{Int}(D).$$
\end{lem}
\begin{lem}
If $f\in C^{m+\alpha}(D)$ and $\mu+\nu=m$, then
$$\T^\nu\overline T^\mu(\partial^\mu\bar\partial^\nu f)=f-\sum_{j=0}^{\nu-1}T^j(S(\bar\partial^j f))-\sum_{j=0}^{\mu-1}T^\nu\overline
T^j(\overline S(\partial^j\bar\partial^\nu f)) \mbox{ on }\mathrm{Int}(D).$$
if $\mu, \nu\geq 1$; otherwise
$$T^m(\bar\partial^m g)=g-\sum_{j=0}^{m-1}T^j(S(\bar\partial^j g)),$$
$$\overline T^m(\partial^m g)=g-\sum_{j=0}^{m-1}\overline T^j(\overline S(\partial^j g)).$$
\end{lem}
\begin{proof}
Just apply Lemma 3.1 repeatedly. In fact, we have two identities.
If $g\in C^{m-\nu+\alpha}(D)$ then
$$T^\nu(\bar\partial^\nu g)=g-\sum_{j=0}^{\nu-1}T^j(S(\bar\partial^j g)).$$
If $g\in C^{m-\mu+\alpha}(D)$ then
$$\overline T^\mu(\partial^\mu g)=g-\sum_{j=0}^{\mu-1}\overline T^j(\overline S(\partial^j g)).$$
\end{proof}
\begin{rem}
This lemma shows $f$ can be represented by integral equations through derivatives of all orders.
We also remark that since $\bar\partial S (f)=\partial\bar S(f)=0$, it holds
$$\partial^\mu\bar\partial^\nu\{\sum_{j=0}^{\nu-1}T^j(S(\bar\partial^j f))+\sum_{j=0}^{\mu-1}T^\nu\overline T^j(\overline S(\partial^j\bar\partial^\nu f))\}=0.$$
\end{rem}
\subsection{A new high order integral operator}
For $k\geq 0$, we define a new operator $^{k+2}Tf$ on $C^{k+\alpha}(D)$.
$$^{k+2}Tf(z)=\frac{-(k+1)!}{2\pi i}\int_D\frac{f(\zeta)-P_{k}(\zeta,z)}{(\zeta-z)^{k+2}}d\bar\zeta\wedge d\zeta,$$
where $P_{k}(\zeta,z)$ is Taylor expansion of $f$ at $z$ of degree $k$. Namely
$$P_{k}(\zeta,z)=\sum_{l=0}^k\frac{1}{l!}\sum_{i+j=l}\partial^i\bar\partial^j f(z)(\zeta-z)^i{(\bar \zeta-\bar z)}^j.$$
We also define $^{k+2}\overline Tf=\overline {^{k+2}T(\bar f)}$.
More generally, if $\Delta$ is a bounded domain, then $^{k+2}T_\Delta$ is  defined as
$$^{k+2}T_\Delta f(z)=\frac{-(k+1)!}{2\pi i}\int_\Delta\frac{f(\zeta)-P_{k}(\zeta,z)}{(\zeta-z)^{k+2}}d\bar\zeta\wedge d\zeta.$$
Note the case $k=0$ was defined in [NW]. The following is well-known and classical(see NW], [V]).
\begin{lem}
Let $f\in C^\alpha(D)$. Then $Tf\in C^{1+\alpha}(D), ^2Tf\in C^\alpha(D)$. Moreover
$$\bar\partial Tf=f,\mbox{  }\partial Tf=^2Tf,$$
$$H_\alpha[^2Tf]\leq C_0 H_\alpha[f],$$
where $C_0=\frac{12}{\alpha(1-\alpha)}$.
If $f\in C^{k+\alpha}(D)(k\geq 0)$, then $Tf\in f\in C^{k+1+\alpha}(D)$.
\end{lem}
We will extend this result to the operator $^{k+2}T$. In order to simplify proofs later, we need several lemmas.
\begin{lem}
If $\varphi(\zeta) $ is holomorphic in $\triangle$, then
$$\int_\triangle\frac{\overline{\phi(\zeta)}}{\zeta-w}d\bar\zeta\wedge d\zeta$$
is anti-holomorphic in $w\in \mathrm{Int}(\triangle)$, where $\triangle=\{|\zeta-z_0|\leq r\}.$
\end{lem}
\begin{proof} Writing $\phi(\zeta)=\sum a_l(\zeta-z_0)^l$, we only have to prove the lemma for $\phi(\zeta)=(\zeta-z_0)^l$. Indeed,
\begin{eqnarray}
\int_{\triangle} \frac{(\bar\zeta-\bar z_0)^l}{\zeta-w}d\bar\zeta\wedge d\zeta &=&{\frac{-2\pi i}{(l+1)}}T_{\triangle}(\bar\partial (\bar\zeta-\bar z_0)^{l+1})(w)\nonumber\\
&=&{\frac{-2\pi i}{(l+1)}}\{(\bar w-\bar z_0)^{l+1}-S_\triangle((\bar\zeta-\bar z_0)^{l+1})(w)\},\nonumber
\end{eqnarray}
where we have
\begin{eqnarray}
S_\triangle((\bar\zeta -\bar z_0)^{l+1})(w) &=& \frac{1}{2\pi i}\int_{|\zeta-z_0|=r}\frac{(\bar\zeta-\bar z_0)^{l+1}}{\zeta-w}d\zeta\nonumber\\
&=& \frac{r^{2(l+1)}}{2\pi i}\int_{|\zeta-z_0|=r}\frac{1}{(\zeta-z_0)^{l+1}(\zeta-w)}d\zeta\nonumber\\
&=&0,\nonumber
\end{eqnarray}
where in the last equality, we have used the residue theorem to get zero integral.
\end{proof}
\begin{cor}It holds for $l\geq 0$,
$$\int_{\triangle} \frac{(\bar\zeta-\bar z_0)^l}{\zeta-w}d\bar\zeta\wedge d\zeta={\frac{-2\pi i}{(l+1)}}(\bar w-\bar z_0)^{l+1}$$
where $\triangle=\{|\zeta-z_0|\leq r\}.$
\end{cor}
\begin{lem}
If $l\geq 1$, then
$$\int_\triangle \frac{(\bar\zeta-\bar z)^l}{(\zeta-z)^{l+1}}d\bar\zeta\wedge d\zeta=0$$
for $z\in \mathrm{Int(\triangle)}.$
\end{lem}
\begin{proof}
We notice that $\frac{(\bar\zeta-\bar z)^l}{(\zeta-z)^{l+1}}$ is integrable in $\triangle$ using polar coordinate at $z$. Let $\epsilon(z)$ be the disk of radius $\epsilon$ and center at $z$ so that $\epsilon(z)\subset\triangle$.
Now we have
\begin{eqnarray}
\int_\triangle\frac{(\bar\zeta-\bar z)^l}{(\zeta-z)^{l+1}}d\bar\zeta\wedge d\zeta &=&
\int_{\triangle\setminus \epsilon(z)} \frac{(\bar\zeta-\bar z)^l}{(\zeta-z)^{l+1}}d\bar\zeta\wedge d\zeta+\int_{\epsilon(z)} \frac{(\bar\zeta-\bar z)^l}{(\zeta-z)^{l+1}}d\bar\zeta\wedge d\zeta\nonumber\\
&=&
\frac{1}{l!}\frac{d^l}{dw^l}\int_{\triangle\setminus \epsilon(z)} \frac{(\bar\zeta-\bar z)^l}{\zeta-w}d\bar\zeta\wedge d\zeta\bigg |_{w=z}+\int_{\epsilon(z)} \frac{(\bar\zeta-\bar z)^l}{(\zeta-z)^{l+1}}d\bar\zeta\wedge d\zeta\nonumber\\
&=& I_1|_{w=z}+I_2\nonumber
\end{eqnarray}
We have
$$I_1=\frac{1}{l!}\frac{d^l}{dw^l}\left\{\int_{\triangle} \frac{(\bar\zeta-\bar z)^l}{\zeta-w}d\bar\zeta\wedge d\zeta-\int_{\epsilon(z)} \frac{(\bar\zeta-\bar z)^l}{\zeta-w}d\bar\zeta\wedge d\zeta\right\}$$
where $w\in \epsilon(z)$. By Lemma 3.5, $I_1=0$. $\lim_{\epsilon\to 0}I_2=0$ is obvious. This proof is complete.
\end{proof}
We prove the following, which is one of key tools in the proof of the results
\begin{thm}
Let $f\in C^{k+\alpha}(D) (k\geq 0)$. Then
$$\partial^i\bar\partial^j Tf=\partial^i\bar\partial^{j-1}f,$$
if $j\geq 1, i+j\leq k+1;$ otherwise
$$\partial^{k+1} Tf=^{k+2}Tf.
$$

\end{thm}
\begin{proof}
For $k=0$, it is Lemma 3.3. Assume $k-1$ is true, i.e., $\partial^{k} Tf=^{k+1}Tf$. To show $k$ is true, it suffices to show that if $f\in C^{k+\alpha}(D)$, then $^{k+1}Tf$ has a total differential, which means we must show that there are numbers $A$ and $B$, depending on $z$, such that
$$^{k+1}Tf(z)-^{k+1}Tf(z')=A(z-z')+B(\bar z-\bar z')+\varepsilon(z,z')\cdot|z-z'|$$
where $\varepsilon (z,z')\to 0$ as $z'\to z$. It is claimed that $A=^{k+2}Tf(z)$ and $B=\partial^k f(z)$. To this end, denote the left side in the above by $I$, and restrict $z'$ so that $\rho=|z-z'|<\frac{1}{2}(R-|z|)$. Let $\delta$ be the disk of radius $\rho$ and center $(z+z')/2$. Notice $\delta\subset D$. Then
$$I=^{k+1}T_\delta f(z)- ^{k+1}T_\delta f(z')-\frac{k!}{2\pi i} \int_{D\setminus\delta}\bigg\{\frac{f(\zeta)-P_{k-1}(\zeta,z)}{(\zeta-z)^{k+1}}-\frac{f(\zeta)-P_{k-1}(\zeta,z')}{(\zeta-z')^{k+1}}\bigg\}d\bar\zeta\wedge d\zeta.$$
We write by Lemma 2.2 that
$$f(\zeta)=P_{k-1}(\zeta, z)+\frac{1}{k!}\sum_{i+j=k}\partial^i\bar\partial^j f(z)(\zeta-z)^i(\bar\zeta-\bar z)^j+E_{k}(\zeta, z)$$
where
$$|E_{k}(\zeta, z)|\leq
\frac{1}{k!}\bigg\{\sum_{i+j=k}H_\alpha[\partial^i\bar\partial^j f]\bigg\}|\zeta-z|^{k+\alpha}.$$

We now consider the terms in order. First we need the following, using Lemma 3.6 and Corollary 3.5,
\begin{eqnarray}
^{k+1}T_\delta f(z)&=&\frac{-k!}{2\pi i}\int_\delta\frac{f(\zeta)-P_{k-1}(\zeta,z)}{(\zeta-z)^{k+1}}d\bar\zeta\wedge d\zeta\nonumber\\
&=&\frac{-1}{2\pi i }\sum_{i+j=k}\partial^i\bar\partial^j f(z)\int_\delta\frac{(\zeta-z)^i(\bar\zeta-\bar z)^j}{(\zeta-z)^{k+1}}d\bar\zeta\wedge d\zeta\nonumber\\
&+&\frac{-k!}{2\pi i}\int_\delta\frac{E_{k}(\zeta,z)}{(\zeta-z)^{k+1}}d\bar\zeta\wedge d\zeta\nonumber\\
&=&\partial^k f(z)\frac{-1}{2\pi i}\int_\delta\frac{1}{(\zeta-z)}d\bar\zeta\wedge d\zeta+\frac{-k}{2\pi i}\int_\delta\frac{E_{k}(\zeta,z)}{(\zeta-z)^{k+1}}d\bar\zeta\wedge d\zeta\nonumber\\
&=&\partial^k f(z)(\bar z-\frac{\bar z+\bar z'}{2})+\frac{-k}{2\pi i}\int_\delta\frac{E_{k}(\zeta,z)}{(\zeta-z)^{k+1}}d\bar\zeta\wedge d\zeta\nonumber\\
&=&\frac{1}{2}\partial^k f(z)(\bar z-\bar z')+I_3.\nonumber
\end{eqnarray}
Similarly we have
\begin{eqnarray}
^{k+1}T_\delta f(z')&=&\frac{1}{2}\partial^k f(z')(\bar z'-\bar z)+\frac{-k}{2\pi i}\int_\delta\frac{E_{k}(\zeta,z')}{(\zeta-z')^{k+1}}d\bar\zeta\wedge d\zeta\nonumber\\
&=&\frac{1}{2}\partial^k f(z')(\bar z'-\bar z)+I_3'.\nonumber
\end{eqnarray}
Now we estimate $I_3, I_3'$.
$$|I_3|=\bigg|\frac{-k}{2\pi i}\int_\delta\frac{E_{k}(\zeta,z)}{(\zeta-z)^{k+1}}d\bar\zeta\wedge d\zeta\bigg|\leq \frac{1}{2\pi (k-1)!}\sum_{i+j=k}H_\alpha[\partial^i\bar\partial^j f]\int_0^{2\pi}\int_0^{r_1(\theta)}\frac{r^{k+\alpha}2rdrd\theta}{r^{k+1}}$$
where $r_1(\theta)$ is the distance from $z$ to the boundary of $\delta$ in direction $\theta$; note $r_1(\theta)\leq2\rho$. Continuing the computation, we have
$$|I_3|\leq \frac{2 }{(k-1)!}\sum_{i+j=k}H_\alpha[\partial^i\bar\partial^j f]\frac{(2\rho)^{1+\alpha}}{1+\alpha}$$
which approaches $0$ as $\rho\to 0$. The same estimate holds true for $I_3'$. Thus we have
\begin{eqnarray}
&&|^{k+1}T_\delta f(z)- ^{k+1}T_\delta f(z')-\partial^kf(z)(\bar z-\bar z')|\nonumber\\
&\leq& |I_3|+|I_3'|+\bigg|\frac{1}{2}\partial^k f(z)(\bar z-\bar z')-\frac{1}{2}\partial^k f(z')(\bar z'-\bar z)-\partial^kf(z)(\bar z-\bar z')\bigg|\nonumber\\
&\leq& |I_3|+|I_3'|+\frac{1}{2}H_\alpha[\partial^k f]|z-z'|^{1+\alpha}.\nonumber
\end{eqnarray}
On the other hand, to continue to estimate, we have
\begin{eqnarray}
&&\int_{D\setminus\delta}\bigg\{\frac{f(\zeta)-P_{k-1}(\zeta,z)}{(\zeta-z)^{k+1}}-\frac{f(\zeta)-P_{k-1}(\zeta,z')}{(\zeta-z')^{k+1}}\bigg\}d\bar\zeta\wedge d\zeta \nonumber\\
&=&\int_{D\setminus\delta}f(\zeta)\bigg(\frac{1}{(\zeta-z)^{k+1}}-\frac{1}{(\zeta-z')^{k+1}}\bigg)d\bar\zeta\wedge d\zeta-
\int_{D\setminus\delta}\frac{P_{k-1}(\zeta, z)}{(\zeta-z)^{k+1}}d\bar\zeta\wedge d\zeta \nonumber\\
&+&\int_{D\setminus\delta}\frac{P_{k-1}(\zeta, z')}{(\zeta-z')^{k+1}}d\bar\zeta\wedge d\zeta\nonumber\\
&=& I_4+I_5+I_6.\nonumber
\end{eqnarray}
We need the following lemma to proceed.
\begin{lem}
If $m\geq 2, n\geq 0$, then
$$\int_{D\setminus\delta}\frac{(\bar\zeta-\bar z)^n}{(\zeta-z)^m}d\bar\zeta\wedge d\zeta=\int_{D\setminus\delta}\frac{(\bar\zeta-\bar z')^n}{(\zeta-z')^m}d\bar\zeta\wedge d\zeta=0.$$
Actually more is true.
$$\int_{D\setminus\delta}\frac{(\bar\zeta-\bar s)^n}{(\zeta-s)^m}d\bar\zeta\wedge d\zeta=0$$
for any $s\in\delta$.
\end{lem}
\begin{proof}
$$\int_{D\setminus\delta}\frac{(\bar\zeta-\bar z)^n}{(\zeta-z)^m}d\bar\zeta\wedge d\zeta=\frac{1}{(m-1)!}\frac{d^{m-1}}{dw^{m-1}}
\int_{D\setminus\delta}\frac{(\bar\zeta-\bar z)^n}{\zeta-w}d\bar\zeta\wedge d\zeta\bigg|_{w=z},$$
which is zero by applying Lemma 3.5.
\end{proof}
By Lemma 3.8, we have
$$I_5=I_6=0.$$
Now we only have to estimate $I_4$. Indeed,
$$I_4=-(k+1)\int_{z'}^z\bigg(\int_{D\setminus\delta}\frac{f(\zeta)}{(\zeta-w)^{k+2}}d\bar\zeta\wedge d\zeta\bigg) dw$$
where the integration with respect to $w$ is along the straight line segment from $z'$ to $z$. Now it suffices to estimate
\begin{eqnarray}
&&\int_{D\setminus\delta}\frac{f(\zeta)}{(\zeta-w)^{k+2}}d\bar\zeta\wedge d\zeta-\int_{D}\frac{f(\zeta)-P_k(\zeta, z)}{(\zeta-z)^{k+2}}d\bar\zeta\wedge d\zeta\nonumber\\
&=&\int_{D\setminus\delta}\frac{f(\zeta)}{(\zeta-w)^{k+2}}d\bar\zeta\wedge d\zeta-
\int_{D\setminus\delta}\frac{f(\zeta)-P_k(\zeta, z)}{(\zeta-z)^{k+2}}d\bar\zeta\wedge d\zeta\nonumber\\
&+&\int_{\delta}\frac{f(\zeta)-P_k(\zeta, z)}{(\zeta-z)^{k+2}}d\bar\zeta\wedge d\zeta\nonumber\\
&=&\int_{D\setminus\delta}f(\zeta)\bigg\{\frac{1}{(\zeta-w)^{k+2}}-\frac{1}{(\zeta-z)^{k+2}}\bigg\}d\bar\zeta\wedge d\zeta+\int_{\delta}\frac{f(\zeta)-P_k(\zeta, z)}{(\zeta-z)^{k+2}}d\bar\zeta\wedge d\zeta\nonumber\\
&=&I_7+I_8.\nonumber
\end{eqnarray}
Here we have used Lemma 3.8 to conclude
$$\int_{D\\delta}\frac{f(\zeta)-P_k(\zeta, z)}{(\zeta-z)^{k+2}}d\bar\zeta\wedge d\zeta=0.$$
Now we have
\begin{eqnarray}
I_7&=&\int_{D\setminus\delta}f(\zeta)\bigg\{\frac{1}{(\zeta-w)^{k+2}}-\frac{1}{(\zeta-z)^{k+2}}\bigg\}d\bar\zeta\wedge d\zeta\nonumber\\
&=&(k+2)\int_{D\setminus\delta}f(\zeta)\int_w^z\frac{ds}{(\zeta-s)^{k+3}}d\bar\zeta\wedge d\zeta\nonumber\\
&=&(k+2)\int_w^z ds \int_{D\setminus\delta}\frac{f(\zeta)-P_{k}(\zeta, s)}{(\zeta-s)^{k+3}}d\bar\zeta\wedge d\zeta\nonumber
\end{eqnarray}
where we have used that
$$\int_{D\setminus\delta}\frac{P_{k}(\zeta, s)}{(\zeta-s)^{k+3}}d\bar\zeta\wedge d\zeta=0,$$
for $s\in\delta$ by Lemma 3.8. Now
\begin{eqnarray}
|I_7|&\leq& (k+2)|w-z|\bigg |\int_{D\setminus\delta}\frac{f(\zeta)-P_{k}(\zeta, s)}{(\zeta-s)^{k+3}}d\bar\zeta\wedge d\zeta\bigg|\nonumber\\
&\leq& \frac{k+2}{k!}|w-z|\sum_{i+j=k}H_\alpha[\partial^i\bar\partial^j f]2\pi\int_{\rho/2}^{2R}\frac{r^{k+\alpha}2rdr}{r^{k+3}}\nonumber\\
&=&2\frac{k+2}{k!}|w-z|\sum_{i+j=k}H_\alpha[\partial^i\bar\partial^j f]2\pi \frac{1}{1-\alpha}((\rho/2)^{\alpha-1}-(2R)^{\alpha-1})\nonumber\\
&\leq& \frac{4\pi(k+2)}{(1-\alpha)k!}\sum_{i+j=k}H_\alpha[\partial^i\bar\partial^j f]\rho^\alpha.\nonumber
\end{eqnarray}
In the last inequality we have used that $|w-z|\leq |z-z'|=\rho.$
We have thus shown that
$$^{k+1}Tf(z)-^{k+1}Tf(z')=^{k+2}Tf(z)(z-z')+\partial^k f(z)(\bar z-\bar z')+O(|z-z'|^{1+\alpha})$$
as $z'\to z$. Thus $^{k+1}Tf$ has a total differential and $\partial (^{k+1}Tf)=^{k+2}Tf$.
\end{proof}
In what follows, we will show that $^{k+2}Tf\in C^\alpha(D)$ if $f\in C^{k+\alpha}(D)$.
\begin{lem}
Let $f\in C^{k+\alpha}(D)$. Then
$$|^{k+2}Tf|\leq C_1(k+1)R^\alpha \sum_{i+j=k}H_\alpha[\partial^i\bar\partial^jf],$$
where $C_1=\frac{2^{\alpha+1}}{\alpha }$.
\end{lem}

\begin{proof}
We compute, using Lemma 2.2,
\begin{eqnarray}
|^{k+2}Tf(z)|&=&\frac{(k+1)!}{2\pi}\bigg|\int_D\frac{f(\zeta)-P_{k}(\zeta,z)}{(\zeta-z)^{k+2}}d\bar\zeta\wedge d\zeta\bigg|,\nonumber\\
&\leq&\frac{(k+1)!}{2\pi k!}\sum_{i+j=k}H_\alpha[\partial^i\bar\partial^jf]\int_0^{2\pi}\int_0^{2R}\frac{r^{k+\alpha}2rdrd\theta}{r^{k+2}}\nonumber\\
&\leq& \frac{2^{\alpha+1}(k+1)}{\alpha }R^\alpha\sum_{i+j=k}H_\alpha[\partial^i\bar\partial^jf].\nonumber
\end{eqnarray}
\end{proof}
\begin{lem}
If $f\in C^{1+\alpha}(D)$, then
$$^2Tf=T(\partial f)- S_b(f).$$
\end{lem}
\begin{proof}
If $f\in C^{1+\alpha}(D)$, then $^2Tf$ is a proper integral. Therefore,
$$-2\pi i^2Tf(z)=\lim_{\varepsilon\to 0}\int_{D_\varepsilon}\frac{f(\zeta)-f(z)}{(\zeta-z)^2}d\bar\zeta\wedge d\zeta=\lim_{\varepsilon\to 0}\int_{D_\varepsilon}\frac{f(\zeta)}{(\zeta-z)^2}d\bar\zeta\wedge d\zeta,$$
where $D_\varepsilon=D\setminus\{|\zeta-z|<\varepsilon\}$, and we have used $\int_{D_\varepsilon}\frac{f(z)}{(\zeta-z)^2}d\bar\zeta\wedge d\zeta=0$ by Lemma 3.8. Now we can apply Stokes theorem to get
$$\int_{D_\varepsilon}\frac{f(\zeta)}{(\zeta-z)^2}d\bar\zeta\wedge d\zeta=\int_{D_\varepsilon}d(\frac{\partial f(\zeta)}{(\zeta-z)}d\bar\zeta)= \int_{D_\varepsilon}\frac{\partial f(\zeta)}{\zeta-z}d\bar\zeta\wedge d\zeta
-\int_{\partial D_\varepsilon}\frac{f(\zeta)}{\zeta-z}d\bar\zeta.$$
We note that
$$\int_{|\zeta-z|=\varepsilon}\frac{f(\zeta)}{\zeta-z}d\bar\zeta=\int_{|\zeta-z|=\varepsilon}\frac{f(\zeta)-f(z)}{\zeta-z}d\bar\zeta,$$
which converges to $0$ as $\varepsilon\to 0$. Here we have used
$$\int_{|\zeta-z|=\varepsilon}\frac{f(z)}{\zeta-z}d\bar\zeta=-\varepsilon^2\int_{|\zeta-z|=\varepsilon}\frac{f(z)}{\zeta^2(\zeta-z)}d\bar\zeta=0,$$
which  is the result of the residue theorem. Letting $\varepsilon\to 0$, the lemma is proved.
\end{proof}
\begin{lem}
If $f\in C^{k+\alpha}(D)$, then for $k\geq 1$
$$^{k+2}Tf=^2T(\partial^k f)-\sum_{i=1}^k\partial^i S_b(\partial^{k-i}f).$$
\end{lem}
\begin{proof}
By Lemma 3.10, we have
$$^2Tf=T(\partial f)- S_b(f).$$
Taking derivative, we have
$$\partial ^2Tf=\partial T(\partial f)- \partial S_b(f),$$
which is equivalent to
$$ ^3Tf=^2T(\partial f)- \partial S_b(f).$$
A mathematical induction finishes the proof.
\end{proof}
\begin{lem}
If $f\in C^{k+\alpha}(D)$, then
$$H_\alpha[\partial^k S_b f]\leq C_2 \frac{R^\alpha}{k+1} \sum_{i+j=k}H_\alpha[\partial^i\bar\partial^j f],$$
where $C_2=\frac{4}{\alpha(1-\alpha)}$.
\end{lem}
\begin{proof}
We notice that
$$\partial^k S_b f(z)=\frac{k!}{2\pi i}\int_C\frac{f(\zeta)}{(\zeta-z)^{k+1}}d\bar\zeta,$$
and $S_b f(z)$ is holomorphic in $\mathrm{Int}(D)$.
First we need the following
$$\int_C\frac{P_k(\zeta, w)}{(\zeta-w)^{k+2}}d\bar\zeta=0,$$
for $w\in \mathrm{Int}(D)$, and where $P_k(\zeta, w)$ is the Taylor expansion of $f(\zeta)$ of order $k$ at $w$. Indeed,
$$\int_C\frac{P_k(\zeta, w)}{(\zeta-w)^{k+2}}d\bar\zeta=\sum_{l=0}^k\frac{1}{l!}\sum_{i+j=l}\partial^i\bar\partial^j f(w)\int_C\frac{(\zeta-w)^i(\bar\zeta-\bar w)^j}{(\zeta-w)^{k+2}}d\bar\zeta.$$
It suffices to prove the following
$$\int_C\frac{(\bar\zeta-\bar w)^n}{(\zeta-w)^{m+2}}d\bar\zeta=0$$
for $m,n\geq 0$. This is equivalent to prove that
$$\int_C\frac{\bar\zeta^n}{(\zeta-w)^{m+2}}d\bar\zeta=0$$
for $m,n\geq 0$. This is equivalent to prove that
$$\int_C\frac{1}{\zeta^{n+2}(\zeta-w)^{m+2}}d\zeta=0.$$
We apply Residue theorem to get zero integral. Consider $\phi(\zeta)=\frac{1}{\zeta^{n+2}(\zeta-w)^{m+2}}$, and it has finite order singularity at
$\zeta=0, w$. The residue of $\phi(\zeta)$ at $\zeta=0$ equals
$$\frac{1}{(n+1)!}[(\zeta-w)^{-(m+2)}]^{(n+1)}\bigg|_{\zeta=0},$$
which is
$$\frac{1}{(n+1)!}(-1)^{m}(m+2)(m+3)\cdot\cdot\cdot(m+n+2)w^{-(m+n+2)}.$$
The residue of $\phi(\zeta)$ at $\zeta=w$ if $w\not=0$ equals
$$\frac{1}{(m+1)!}[\zeta^{-(n+2)}]^{(m+1)}\bigg|_{\zeta=w},$$
which is
$$\frac{1}{(m+1)!}(-1)^{m+1}(n+2)(n+3)\cdot\cdot\cdot(m+n+2)w^{-(m+n+2)}.$$
They are opposite so the integral is zero. The case $w=0$ is obvious. Now we are ready to estimate.
\begin{eqnarray}
&&\frac{2\pi i}{k!}(\partial^k S_b f(z)-\partial^k S_b f(z'))\nonumber\\ &=&\int_Cf(\zeta)\bigg[\frac{1}{(\zeta-z)^{k+1}}-\frac{1}{(\zeta-z')^{k+1}}\bigg]d\bar\zeta\nonumber\\
&=&\frac{1}{k+1}\int_Cf(\zeta)\int_z^{z'}\frac{dw}{(\zeta-w)^{k+2}}d\bar\zeta\nonumber\\
&=&\frac{1}{k+1}\int_z^{z'}dw\int_C \frac{f(\zeta)}{(\zeta-w)^{k+2}}d\bar\zeta\nonumber\\
&=&\frac{1}{k+1}\int_z^{z'}dw\int_C \frac{f(\zeta)-P_k(\zeta,w)}{(\zeta-w)^{k+2}}d\bar\zeta.\nonumber
\end{eqnarray}
In order to apply estimates in [NW], we convert integral to be on unit circle. Setting $w=R\omega$ and $\zeta=R\eta$, this becomes
$$\frac{1}{(k+1)R^k}\int_{z/R}^{z'/R}d\omega\int_{|\eta|=1}\frac{f(R\eta)-P_k(R\eta,R\omega)}{(\eta-\omega)^{k+2}}d\bar\eta$$
where the line integral is taken on the shorter segment of the circle through $z/R$ and $z'/R$ and orthogonal to the unit circle(see [NW] (6.2a)).
Further, setting $\tau=(\eta-\omega)/(1-\bar\omega\eta)$, we have
$$\eta=\frac{\tau+\omega}{1+\bar\omega\tau}, \mbox{  }\eta-\omega=\frac{\tau(1-\omega\bar\omega)}{1+\tau\bar\omega},\mbox{  } d\eta=\frac{(1-\omega\bar\omega)}{(1+\tau\bar\omega)^2}d\tau.$$
Hence
\begin{eqnarray}
&\frac{2\pi }{k!}&|\partial^k S_b f(z)-\partial^k S_b f(z') |\nonumber\\
&\leq &\frac{1}{(k+1)R^k}\int_{z/R}^{z'/R}|d\omega|\int_{|\eta|=1}\bigg|\frac{f(R\eta)-P_k(R\eta,R\omega)}{(\eta-\omega)^{k+2}}\bigg||d\bar\eta|\nonumber\\
&\leq& \frac{1}{(k+1)R^k}\int_{z/R}^{z'/R}|d\omega|\int_{|\eta|=1}\frac{|f(R\frac{\tau+\omega}{1+\bar\omega\tau})-P_k(R\frac{\tau+\omega}{1+\bar\omega\tau},R\omega)||1+\tau\bar\omega|^k}{(1-\omega\bar\omega)^{k+1}}|d\eta|\nonumber\\
&\leq& \frac{1}{(k+1)k!R^k}\sum_{i+j=k}H_\alpha[\partial^i\bar\partial^j f]\int_{z/R}^{z'/R}|d\omega|\int_{|\eta|=1}\frac{|R\frac{\tau+\omega}{1+\bar\omega\tau}-R\omega|^{k+\alpha}
|1+\tau\bar\omega|^k}{(1-\omega\bar\omega)^{k+1}}|d\eta|\nonumber\\
&\leq&\frac{R^\alpha}{(k+1)!}\sum_{i+j=k}H_\alpha[\partial^i\bar\partial^j f]\int_{z/R}^{z'/R}\frac{|d\omega|}{(1-\omega\bar\omega)^{1-\alpha}}\int_{|\tau|=1}\frac{|d\tau|}{|1+\bar\omega\tau|^\alpha}.\nonumber
\end{eqnarray}
According to [NW](6.2a and p473), one has
$$\int_{|\tau|=1}\frac{|d\tau|}{|1+\bar\omega\tau|^\alpha}\leq \frac{4\pi}{1-\alpha},$$
$$\int_{z/R}^{z'/R}\frac{|d\omega|}{(1-\omega\bar\omega)^{1-\alpha}}\leq \frac{2}{\alpha}|z-z'|^\alpha.$$
Hence
$$H_\alpha[\partial^k S_bf]\leq \frac{4}{\alpha(1-\alpha)}\frac{R^\alpha}{k+1}\sum_{i+j=k}H_\alpha[\partial^i\bar\partial^j f].$$
\end{proof}
\begin{thm}
If $f\in C^{k+\alpha}(D)$, then
$$H_\alpha[^{k+2}Tf]\leq (C_0+kC_2R^\alpha)\sum_{i+j=k}H_\alpha[\partial^i\bar\partial^j f],$$
$$H_\alpha[^{k+2}\overline Tf]\leq (C_0+kC_2R^\alpha)\sum_{i+j=k}H_\alpha[\partial^i\bar\partial^j f].$$
\end{thm}
\begin{proof}
By Lemma 3.11, and applying Lemma 3.12, we have
\begin{eqnarray}
&&^{k+2}Tf=^2T(\partial^k f)-\sum_{i=1}^k\partial^i S_b(\partial^{k-i}f).\nonumber\\
H_{\alpha}[^{k+2}Tf]&\leq& H_\alpha[^2T(\partial^k f)]+\sum_{i=1}^kH_\alpha[\partial^i S_b(\partial^{k-i}f)]\nonumber\\
&\leq& C_0 H_\alpha[\partial^k f]+C_2R^\alpha\left\{\sum_{i=1}^k \frac{1}{i}\sum_{p+q=i}H_\alpha[\partial^p\bar\partial^q\partial^{k-i} f]\right\}\nonumber\\
&\leq& (C_0+C_2 k R^\alpha)\sum_{i+j=k}H_\alpha[\partial^i\bar\partial^j f].\nonumber
\end{eqnarray}
\end{proof}

\begin{lem}
If $h\in C^{k+\alpha}(D)$ $(k\geq 0)$, then
$$\|Th\|^{(k+1)}\leq 2^k(C_1(k+1)+C_0+kC_2R^\alpha) \|h\|^{(k)},$$
$$\|\overline Th\|^{(k+1)}\leq 2^k(C_1(k+1)+C_0+kC_2R^\alpha) \|h\|^{(k)}.$$
\end{lem}
\begin{proof}
Let $i+j=k+1$. If $j\geq 1$, then $\partial^i\bar\partial^j Th=\partial^i\bar\partial^{j-1}h$.
Otherwise, $\partial^{k+1}Th=^{k+2}Th$. We have by Lemma 3.9, 3.13
\begin{eqnarray}
\|^{k+2}Th\|&=&|^{k+2}Th|+(2R)^\alpha H_\alpha[^{k+2}Th]\nonumber\\
&\leq&C_1(k+1)R^\alpha \sum_{i+j=k}H_\alpha[\partial^i\bar\partial^jh]+(2R)^\alpha(C_0+kC_2R^\alpha)\sum_{i+j=k}H_\alpha[\partial^i\bar\partial^j h]\nonumber\\
&\leq&(C_1(k+1)R^\alpha +(2R)^\alpha(C_0+kC_2R^\alpha))\sum_{i+j=k}H_\alpha[\partial^i\bar\partial^j h]\nonumber\\
&\leq&2^k(C_1(k+1)2^{-\alpha}+C_0+kC_2R^\alpha)\|h\|^{(k)}.\nonumber
\end{eqnarray}
In the last inequality we have used that $H_\alpha[\partial^i\bar\partial^j h]\leq (2R)^{-\alpha}\|h\|^{(k)}$ for $i+j=k$. Also notice $2^k(C_1(k+1)2^{-\alpha}+C_0+kC_2R^\alpha)\geq 1$.
The proof is now complete.
\end{proof}
The following is the main result of this section.
\begin{thm}
If $h\in C^{\alpha}(D)$ and $\mu+\nu=m$ , then
$$\|T^\nu\overline T^\mu h\|^{(m)}\leq 2^{\frac{(m-1)m}{2}}(C_1 m+C_0+(m-1)C_2 R^\alpha)^m \|h\|.$$

\end{thm}
\begin{proof}
Since, if $\nu\geq 1$,
$$\|T^\nu\overline T^\mu h\|^{(m)}=\|T(T^{\nu-1}\overline T^\mu h)\|^{(m-1+1)},$$
which is less than, by invoking Lemma 3.14
$$2^{m-1}(C_1m+C_0+{(m-1)}C_2R^\alpha)\|T^{\nu-1}\overline T^\mu h\|^{(m-1)},$$
which is less than, by repeating the argument,
$$\prod_{k=1}^m2^{k-1}(C_1k+C_0+{(k-1)}C_2R^\alpha)\| h\|,$$
which, when simplified, gives the proof of the theorem.
\end{proof}

\section{An High order integral operator }\label{sec1}
In order to prove the existence of smooth solutions of system considered we will solve the equation with initial values at the origin. Namely, we are to solve the system of equations
\begin{eqnarray}
\partial^\mu\bar\partial^\nu u(z) &=& a(z, u(z), \mathcal{D}^1 u(z), ...,\mathcal{D}^{m-1} u(z), \mathcal{D}^m u(z))
\end{eqnarray}
with initial conditions at the origin:
\begin{eqnarray}
 \mathcal{D}^ku(0)&=& 0, \mbox{   }k=0,1,..., m-1\\
\mathcal{D}^m u(0)&=& v \mbox{  except for  } \mu,\nu.
\end{eqnarray}
where $v$ is a nonzero vector in $\co^{2^mn}$ $(|v|\not=0)$. Of course, we have to mention that $v$ must be given so that it matches with the definition of $\mathcal{D}^m u(0)$ (there are mixed derivatives repeatedly appeared).
We define, according to Lemma 3.2
$$\psi(z)=\sum_{j=0}^{\nu-1}T^j(S(\bar\partial^j f)+\sum_{j=0}^{\mu-1}T^\nu\overline T^j(\overline S(\partial^j\bar\partial^\nu f).$$
It is easy to see $\partial^\mu\bar\partial^\nu\psi(z)=0$. Every solution of (15) satisfies the integral equation, by Lemma 3.2 again
\begin{eqnarray}
u(z)=\psi(z) +T^\nu\overline T^\mu a(\zeta, u, \mathcal{D}^1 u, ...,\mathcal{D}^{m-1} u, \mathcal{D}^m u)(z)
\end{eqnarray}
In what follows, we will introduce function spaces that are Banach with norms to reflect the initial conditions, and
will modify the equation (18) to fit the defined Banach space and to apply fixed point theorem to produce the needed solutions. In doing so, we will take full advantages of
two parameters: the radius of $D$ and the radius of a closed ball in the Banach space where the solutions are sought from.

\subsection{Iteration procedures} For the sake of generality we consider a Banach space $\mathbf{B}(R)$ depending on a positive parameter $R$. $R$ will be less than some fixed positive quantity.
The norm on $\mathbf{B}(R)$ is denoted by $\|\cdot\cdot\cdot\|$; it may also depend on $R$.
The following deals with a map $\mathbf{\Theta}:\mathbf{A}(R,\gamma)\to \mathbf{B}(R)$, where $\mathbf{A}(R,\gamma)(\gamma>0)$ is a closed subset of $\mathbf{B}(R)$ defined as
$\mathbf{A}(R,\gamma)=\{\Omega|\|\Omega\|\leq \gamma\}$. Here we actually have a family of such maps with two parameters $R,\gamma$.
\begin{lem}
 Let $\mathbf{\Theta}:\mathbf{A}(R,\gamma)\to \mathbf{B}(R)$, and
and let $\delta(R,\gamma)$ and $\eta(R,\gamma)$ exist such that for all $\Omega,\Omega'\in \mathbf{A}(R,\gamma)$
$$\|\mathbf{\Theta}(\Omega')-\mathbf{\Theta}(\Omega)\|\leq \delta(R,\gamma)\|\Omega'-\Omega\|,$$
$$\|\mathbf{\Theta}(\Omega)\|\leq \eta(R,\gamma).$$
If there exist $R_0>0, \gamma_0>0$ such that  $\delta(R_0,\gamma_0)\leq 3/4$ and $\eta(R_0,\gamma_0)\leq \frac{\gamma_0}{2}$. Let $\psi\in\mathbf{B}(R)$ be such that
$\|\psi\|\leq \frac{\gamma_0}{2}$. Then the following equation
$$ \Omega=\psi+\mathbf{\Theta}(\Omega)$$
has a unique solution in $\mathbf{A}(R,\gamma_0)$ for $R\leq R_0$. The solution $\Omega$ is the limit of the sequence
\begin{eqnarray}
\Omega_{N+1}=\psi+\mathbf{\Theta}(\Omega_N)\mbox{  for } N=1,2,...\nonumber
\end{eqnarray}
where $\Omega_1\in \mathbf{A}(R_0,\gamma_0)$.
\end{lem}
\begin{rem}
We will make a repeated use of this lemma in various cases. The proof is simple consequence of contraction principle.
\end{rem}
\subsection{Integral equations on Banach spaces}
Here we fix the Banach space to work on.
Let $\mathbf{B}(R)$ be $[C^{m+\alpha}_0(D)]^n$, which is, by Lemma 2.8 a Banach space with norm $||\cdot\cdot\cdot||^{(m)}$ defined as follows:
if $f=(f_1,..., f_n)\in \mathbf{B}(R)$, then $\|f\|^{(m)}=\max_{1\leq i\leq n}\|f_i\|^{(m)}$.
Let $\gamma$ be a positive number, and we consider $\mathbf{A}(R,\gamma)=\{ f\in \mathbf{B(\R)}|\|f\|^{(m)}\leq \gamma\}$. $\mathbf{A}(R,\gamma)$ is a closed subset of $\mathbf{B}(R)$.
Now we assume  $u=(u^1,...,u^n)$ and $a=(a^1,...,a^n)$. According to (18) , we consider equations
$$u^i=\psi^i+T^\nu\overline T^\mu a^i(z, u, \mathcal{D}^1 u, ...,\mathcal{D}^{m-1} u, \mathcal{D}^m u)$$
where $\psi=(\psi^1,...,\psi^n)$ is such that $\partial^\mu\bar\partial^\nu\partial \psi=0$.
To simplify the notation we define
$$\omega^i(f)=T^\nu\overline T^\mu a^i(z, f, \mathcal{D}^1 f, ...,\mathcal{D}^{m-1} f, \mathcal{D}^m f)$$
for all $f\in \mathbf{B}(R)$ and $i=1,2,...,n$.
To assure that the initial conditions in (16) (17) are also satisfied , and to further fix the solution, we consider
\begin{eqnarray}
\mathbf{\Theta}^i(f)(\zeta)&=&\omega^i(f)(\zeta)-\sum_{p=0}^{m-1}\frac{1}{p!}\sum_{k+l=p}[\partial^k\bar\partial^l \omega^i(f)](0)\zeta^k\bar\zeta^l\nonumber\\
&-&\frac{1}{m!}\sum_{k+l=m}[\partial^k\bar\partial^l \omega^i(f)](0)\zeta^k\bar\zeta^l+\frac{1}{\mu!\nu!}[\partial^\mu\bar\partial^\nu \omega^i(f)](0)\zeta^\mu\bar\zeta^\nu.
\end{eqnarray}

We first notice that $T, \overline T $ map $C^\alpha(\D)\to C^{1+\alpha}(D)$, and it follows, by the construction (19), $\mathbf{\Theta}^i(f)(\zeta)\in
C^{m+\alpha}_0(D)$ for $f\in \mathbf{B}(R)$ and $[\partial^k\bar\partial^l{\Theta}^i(f)](0)=0$ for $k+l=m, k\not=\mu,l\not=\nu$ . Thus we define a map from $\mathbf{B}(R)$ to $\mathbf{B}(R)$.
\begin{eqnarray}
\mathbf{\Theta}&:&\mathbf{B}(R)\rightarrow\mathbf{B}(R)\nonumber\\
\mathbf{\Theta}(f)&=&(\mathbf{\Theta}^1(f),...,\mathbf{\Theta}^n(f)).
\end{eqnarray}
Here we recall again $\|\mathbf{\Theta}(f)\|^{(m)}=\max_{1\leq i\leq n}\|\mathbf{\Theta}^i(f)\|^{(m)}$. In order to apply iteration procedure,
we will first estimate
$$\|\mathbf{\Theta}^i(f)-\mathbf{\Theta}^i(g)\|^{(m)}$$
for all $f, g\in \mathbf{A}(R,\gamma)$ in terms of $\|f-g\|^{(m)}$. Recall that for $f\in C^{m+\alpha}_0(D)$,
$\|f\|^{(m)}=\max_{k+l=m}\{\|\partial^k\bar\partial^lf\|\}$.
First we have
\begin{eqnarray}
&&\|\mathbf{\Theta}^i(f)-\mathbf{\Theta}^i(g)\|^{(m)}\nonumber\\
&\leq& \|\mathbf{\omega}^i(f)-\mathbf{\omega}^i(g)\|^{(m)}+\bigg|\sum_{k+l=m}\{[\partial^k\bar\partial^l \omega^i(f)](0)
-[\partial^k\bar\partial^l \omega^i(g)](0)\}\bigg|\nonumber\\
&=&I_9+I_{10}.
\end{eqnarray}
Next we want to estimate $\|\mathbf{\Theta}^i(f)\|^{(m)}$ in terms of a constant. In fact, we have the following estimate:
\begin{eqnarray}
\|\mathbf{\Theta}^i(f)\|^{(m)}&\leq& \|\mathbf{\omega}^i(f)\|^{(m)}+\bigg|\sum_{k+l=m}[\partial^k\bar\partial^l\omega^i(f)](0)\bigg|\nonumber\\
&=&I_{11}+I_{12}.\nonumber
\end{eqnarray}
In the following subsections we will provide details of estimates of $I_9, I_{10}, I_{11}$, and $I_{12}$. This is where assumptions are to be used.

\subsubsection{Estimate of $I_9$}
In this section we always have functions $f,g$ to be in $\mathbf{A}(R,\gamma)$.
We first begin with the estimate:
\begin{eqnarray}
&&\|\mathbf{\omega}^i(f)-\mathbf{\omega}^i(g)\|^{(m)}\nonumber\\
&=&\|T^\nu\overline T^\mu (a^i(z, f, \mathcal{D}^1 f, ...,\mathcal{D}^{m} f)-a^i(z, g, \mathcal{D}^1 g, ...,\mathcal{D}^{m} g))\|^{(m)}\nonumber\\
&\leq& M \|a^i(z, f, \mathcal{D}^1 f, ...,\mathcal{D}^{m} f)-a^i(z, g, \mathcal{D}^1 g, ...,\mathcal{D}^{m}g)\|,
\end{eqnarray}
where $M$ is given, according to Theorem 3.15, by
\begin{eqnarray}
M=2^{\frac{(m-1)m}{2}}(C_1 m+C_0+(m-1)C_2 R^\alpha)^m.
\end{eqnarray}
In order to estimate (22), we first estimate the norm $|\cdot\cdot\cdot|$. In the process we want to separate variables in $\mathcal{D}^{m}$ from those in
$\mathcal{D}^0, ..., \mathcal{D}^{m-1}$. In fact, we have
\begin{eqnarray}
&&a^i(\zeta, \mathcal{D}^0f(\zeta), \mathcal{D}^1 f(\zeta), ...,\mathcal{D}^{m} f(\zeta))-a^i(\zeta, \mathcal{D}^0g(\zeta), \mathcal{D}^1 g(\zeta), ...,\mathcal{D}^{m-1} g(\zeta))\nonumber\\
&=&\int_0^1\frac{d}{dt}a^i(\zeta, tf(\zeta)+(1-t)g(\zeta),..., t\mathcal{D}^{m} f(\zeta)+(1-t)\mathcal{D}^{m} g(\zeta))dt\nonumber\\
&=&\sum_{j=0}^n\sum_{p=0}^{m-1}\sum_{k+l=p}A^{p,j}_{k,l}\partial^k\bar\partial^l(f_j-g_j)+\bar A^{p,j}_{k,l}\overline{\partial^k\bar\partial^l(f_j-g_j)}\nonumber\\
&+&\sum_{j=0}^n\sum_{k+l=m}B^{m,j}_{k,l}\partial^k\bar\partial^l(f_j-g_j)+\bar B^{m,j}_{k,l}\overline{\partial^k\bar\partial^l(f_j-g_j)},
\end{eqnarray}
where
$$A^{p,j}_{k,l}=\int_0^1\frac{\partial a^i}{\partial \eta_{p,j}^{k,l}}(\zeta, \mathcal{W}^0,...,\mathcal{W}^{m})dt,\mbox{   }
\bar A^{p,j}_{k,l}=\int_0^1\frac{\partial a^i}{\bar\partial \eta_{p,j}^{k,l}}(\zeta,  \mathcal{W}^0,...,\mathcal{W}^{m})dt,$$
$$B^{m,j}_{k,l}=\int_0^1\frac{\partial a^i}{\partial \eta_{m,j}^{k,l}}(\zeta,  \mathcal{W}^0,...,\mathcal{W}^{m})dt,\mbox{   }
\bar B^{m,j}_{k,l}=\int_0^1\frac{\partial a^i}{\bar\partial \eta_{m,j}^{k,l}}(\zeta,  \mathcal{W}^0,...,\mathcal{W}^{m})dt,$$
where we have used shorten notations:
$$\mathcal{W}^{k}=t\mathcal{D}^k f(\zeta)+(1-t)\mathcal{D}^k g(\zeta)$$
for $k=0, 1, ..., m$, and $\eta_{p,j}^{k,l}$ is a variable in $\eta_p$.
Therefore, taking norm $\|\cdot\cdot\cdot\|$ on (24) we have, using Lemma 2.4,
\begin{eqnarray}
&&\|a^i(\zeta, \mathcal{D}^0f, \mathcal{D}^1 f, ...,\mathcal{D}^{m} f)-a^i(\zeta, \mathcal{D}^0g, \mathcal{D}^1 g, ...,\mathcal{D}^{m} g)\|\nonumber\\
&\leq&\sum_{j=0}^n\sum_{p=0}^{m-1}\sum_{k+l=p}\|A^{p,j}_{k,l}\partial^k\bar\partial^l(f_j-g_j)\|+\|\bar A^{p,j}_{k,l}\overline{\partial^k\bar\partial^l(f_j-g_j)}\|\nonumber\\
&+&\sum_{j=0}^n\sum_{k+l=p}\|B^{m,j}_{k,l}\partial^k\bar\partial^l(f_j-g_j)\|+\|\bar B^{m,j}_{k,l}\overline{\partial^k\bar\partial^l(f_j-g_j)}\|\nonumber\\
&\leq&\sum_{j=0}^n\sum_{p=0}^{m-1}\sum_{k+l=p}\|A^{p,j}_{k,l}\|\|\partial^k\bar\partial^l(f_j-g_j)\|+\|\bar A^{p,j}_{k,l}\|\|\overline{\partial^k\bar\partial^l(f_j-g_j)}\|\nonumber\\
&+&\sum_{j=0}^n\sum_{k+l=p}\|B^{m,j}_{k,l}\|\|\partial^k\bar\partial^l(f_j-g_j)\|+\|\bar B^{m,j}_{k,l}\|\|\overline{\partial^k\bar\partial^l(f_j-g_j)}\|\nonumber\\
&\leq&\sum_{p=0}^{m-1}\sum_{k+l=p}\{\|A^{p,j}_{k,l}\|+\|\bar A^{p,j}_{k,l}\|\}\|{\partial^k\bar\partial^l(f-g)}\|\nonumber\\
&+&\sum_{k+l=p}\{\|B^{m,j}_{k,l}\|+\|\bar B^{m,j}_{k,l}\|\}\|{\partial^k\bar\partial^l(f-g)}\|\nonumber\\
&\leq&\sum_{p=0}^{m-1}\sum_{k+l=p}\{\|A^{p,j}_{k,l}\|+\|\bar A^{p,j}_{k,l}\|\}\frac{6^{m-p}}{(m-p)!}R^{m-p}\|f-g\|^{(m)}\nonumber\\
&+&\sum_{k+l=p}\{\|B^{m,j}_{k,l}\|+\|\bar B^{m,j}_{k,l}\|\}\|f-g\|^{(m)}\nonumber\\
&\leq&\bigg\{\sum_{p=0}^{m-1}\frac{6^{m-p}}{(m-p)!}R^{m-p}\sum_{k+l=p}\{\|A^{p,j}_{k,l}\|+\|\bar A^{p,j}_{k,l}\|\}\nonumber\\
&+&\sum_{k+l=m}\{\|B^{m,j}_{k,l}\|+\|\bar B^{m,j}_{k,l}\|\}\bigg\}\|f-g\|^{(m)}.
\end{eqnarray}
In order to estimate the quantities $\|A^{p,j}_{k,l}\|, \bar A^{p,j}_{k,l}\|, \|B^{m,j}_{k,l}\| $ and $\|\bar B^{m,j}_{k,l}\|$ in (25), we need to study the ranges of $\mathcal{W}^k (k=0, 1, ..., m)$ for $f, g\in \mathbf{A}(R,\gamma)$. The following is what we need.
\begin{lem}
If $f, g\in \mathbf{A}(R,\gamma)$, then
\begin{eqnarray}
|\mathcal{W}^k|&\leq& 6^m R^{m-k}\gamma,\mbox{   } k=0, 1, ..., m-1,\nonumber\\
|\mathcal{W}^m|&\leq& \gamma.
\end{eqnarray}
\end{lem}
\begin{proof}
If $f$ is a function in $C_0^{m+\alpha}(D)$, then by Lemma 2.4
we have, if $i+j=k\leq m-1$,
$$\|\partial^i\bar\partial^j f\|\leq \frac{6^{m-k}}{(m-k)!}R^{m-k}\|f\|^{(m)}\leq \frac{6^{m-k}}{(m-k)!}R^{m-k}\gamma.$$
Particularly, it implies
$$|\partial^i\bar\partial^j f|\leq \frac{6^{m-k}}{(m-k)!}R^{m-k}\gamma,$$
whence, by the definition of norm, we have
$$|\mathcal{W}^k|\leq \frac{6^{m-k}}{(m-k)!}R^{m-k}\gamma\leq 6^mR^{m-k}\gamma$$
for $k=0, 1, ..., m-1$. The result for $|\mathcal{W}^m|$ is obvious.
\end{proof}
To continue on the estimates, we need a lemma on Lipschitz properties of $C_0^{m+\alpha}(D)$.
\begin{lem}
Let $f\in C_0^{m+\alpha}(D)$. Then, for $\zeta', \zeta\in D$, and $i+j=l\leq m-1$, we have
$$|\partial^i\bar\partial^j f(\zeta')-\partial^i\bar\partial^j f(\zeta)|\leq 6^m R^{m-l-1}\|f\|^{(m)}|\zeta'-\zeta|.$$
\end{lem}
\begin{proof} We have
\begin{eqnarray}
\partial^i\bar\partial^j f(\zeta')-\partial^i\bar\partial^j f(\zeta)&=&\int_0^1\frac{d}{dt}\{\partial^i\bar\partial^j f(t\zeta'+(1-t)\zeta)\}dt\nonumber\\
&=&\int_0^1\partial^{i+1}\bar\partial^j f(t\zeta'+(1-t)\zeta)(\zeta'-\zeta)+\partial^{i}\bar\partial^{j+1} f(t\zeta'+(1-t)\zeta)\overline {(\zeta'-\zeta)}dt\nonumber
\end{eqnarray}
whence
\begin{eqnarray}
|\partial^i\bar\partial^j f(\zeta')-\partial^i\bar\partial^j f(\zeta)|&\leq& (|\partial^{i+1}\bar\partial^j f|+|\partial^{i}\bar\partial^{j+1} f|)|\zeta'-\zeta|\nonumber\\
&\leq& 2\|f\|^{(l+1)}|\zeta'-\zeta|\nonumber\\
&\leq&2\frac{6^{m-l-1}}{(m-l-1)!}R^{m-l-1}\|f\|^{(m)}\nonumber\\
&\leq& 6^m R^{m-l-1}\|f\|^{(m)},\nonumber
\end{eqnarray}
where in the last inequality we have used Lemma 2.4.
\end{proof}
In the following we will use a compact set in $\Omega$ defined as follows:
$$E(R, \gamma)=D\times\Pi_{k=0}^{m-1}\{z\in\co^{n2^k}||z|\leq 6^mR^{m-k}\gamma\}\times\{z\in \co^{n2^m}||z|\leq\gamma\}.$$
In order for $a(\zeta, \mathcal{W}^0, ...,  \mathcal{W}^{m-1},  \mathcal{W}^m)$ is defined for $f, g\in [C_0^{m+\alpha}(D)]^n$ with $\|f\|^{(m)}\leq \gamma, \|g\|^{(m)}\leq \gamma$, we now have to choose $R$ or $\gamma$ or both  so that
\begin{eqnarray}
6^mR^{m}\gamma\leq R'.
\end{eqnarray}
We will continue to make choice of $R, \gamma$ under this condition (27). Now we first define some constants, whose usefulness will be self-evident during proofs below and next sections.
\begin{eqnarray}
A(R, \gamma)&=&\max\bigg\{\bigg|\frac{\partial a^i}{\partial \eta_{p,j}^{k,l}}\bigg|_{E(R,\gamma)},\bigg|\frac{\partial a^i}{\bar\partial \eta_{p,j}^{k,l}}\bigg|_{E(R,\gamma)} \bigg |k+l=p; 0\leq p\leq m-1, i,j=1, ..., n\bigg\}\nonumber\\
B(R, \gamma)&=&\max\bigg\{\bigg|\frac{\partial a^i}{\partial \eta_{m,j}^{k,l}}\bigg|_{E(R,\gamma)},\bigg|\frac{\partial a^i}{\bar\partial \eta_{m,j}^{k,l}}\bigg|_{E(R,\gamma)} \bigg |k+l=m;  i=1, ..., n\bigg\}\nonumber\\
C(R,\gamma)&=&\max\bigg\{\bigg|\frac{\partial a^i}{\partial\zeta}\bigg|_{E(R,\gamma)}\bigg|\frac{\partial a^i}{\partial\bar\zeta}\bigg|_{E(R,\gamma)}\bigg|i=1,...,n\bigg\}\nonumber\\
H_\alpha^A[R,\gamma]&=&\max\bigg\{H_\alpha\bigg [\frac{\partial a^i}{\partial \eta_{p,j}^{k,l}}\bigg]_{E(R,\gamma)}, H_\alpha\bigg [\frac{\partial a^i}{\partial \bar\eta_{p,j}^{k,l}}\bigg]_{E(R,\gamma)}\bigg |k+l=p; 0\leq p\leq m-1, i,j=1, ..., n\bigg\}\nonumber\\
H_1^A[R,\gamma]&=&\max\bigg\{H_1\bigg [\frac{\partial a^i}{\partial \eta_{p,j}^{k,l}}\bigg]_{E(R,\gamma)}, H_1\bigg [\frac{\partial a^i}{\partial \bar\eta_{p,j}^{k,l}}\bigg]_{E(R,\gamma)}\bigg |k+l=p; 0\leq p\leq m-1, i,j=1, ..., n\bigg\}\nonumber\\
H_\alpha^B[R,\gamma]&=&\max\bigg\{H_\alpha\bigg [\frac{\partial a^i}{\partial \eta_{m,j}^{k,l}}\bigg]_{E(R,\gamma)}, H_\alpha\bigg [\frac{\partial a^i}{\partial \bar\eta_{m,j}^{k,l}}\bigg]_{E(R,\gamma)}\bigg |k+l=m, i,j=1, ..., n\bigg\}\nonumber\\
H_1^B[R,\gamma]&=&\max\bigg\{H_1\bigg [\frac{\partial a^i}{\partial \eta_{m,j}^{k,l}}\bigg]_{E(R,\gamma)}, H_1\bigg [\frac{\partial a^i}{\partial \bar\eta_{m,j}^{k,l}}\bigg]_{E(R,\gamma)}\bigg |k+l=m, i,j=1, ..., n\bigg\}\nonumber\\
H_\alpha^C[R,\gamma]&=&\max\bigg\{H_\alpha\bigg [\frac{\partial a^i}{\partial\zeta}\bigg]_{E(R,\gamma)}, H_\alpha\bigg [\frac{\partial a^i}{\partial \bar\zeta}\bigg]_{E(R,\gamma)}\bigg | i=1, ..., n\bigg\}\nonumber\\
H_1^C[R,\gamma]&=&\max\bigg\{H_1\bigg [\frac{\partial a^i}{\partial\zeta}\bigg]_{E(R,\gamma)}, H_1\bigg [\frac{\partial a^i}{\partial \bar\zeta}\bigg]_{E(R,\gamma)}\bigg | i=1, ..., n\bigg\}.
\end{eqnarray}
Here we denote $H_1[f]$ for Lipschitz constant of of a function $f$, which is function of many variables. We make a remark here about H\"older constant in many variables and one variable slice. Let $h$ be a function defined on open set $\Omega$ in $\co^N$. The H\"older constant of $h$ is defined
$$H_\alpha[h]=\sup\bigg\{\frac{|h(u)-h(v)|}{|u-v|^\alpha}\bigg|u,v\in \Omega\bigg\}.$$
Now if we restrict $h$ to be on one variable: $g(\zeta)=h(u_1,\cdot\cdot\cdot,u_i,\zeta,u_{i+1},\cdot\cdot\cdot,u_N)$, then
$H_\alpha[g]\leq H_\alpha[h]$. This fact is used below and elsewhere without mention.
It is obvious that
\begin{eqnarray}
&&\left |A^{p,j}_{k,l}\right|\leq A(R,\gamma)
,\mbox{  }\left |\bar A^{p,j}_{k,l}\right |\leq A(R,\gamma),\nonumber\\
&&\left |B^{m,j}_{k,l}\right |\leq B(R,\gamma)
,\mbox{  }\left |\bar B^{m,j}_{k,l}\right |\leq B(R,\gamma).
\end{eqnarray}
Now we estimate: inserting terms,
\begin{eqnarray}
&&A^{p,j}_{k,l}(\zeta')-A^{p,j}_{k,l}(\zeta)\nonumber\\
&=&\int_0^1\{\frac{\partial a^i}{\partial \eta_{p,j}^{k,l}}(\zeta', \mathcal{W}(\zeta'),...,\mathcal{W}^{m-1}(\zeta'),\mathcal{W}^m(\zeta') )-\frac{\partial a^i}{\partial \eta_{p,j}^{k,l}}(\zeta, \mathcal{W}^0(\zeta),...,\mathcal{W}^{m-1}(\zeta),\mathcal{W}^m(\zeta) )\}dt\nonumber\\
&=&\int_0^1\{\frac{\partial a^i}{\partial \eta_{p,j}^{k,l}}(\zeta', \mathcal{W}(\zeta'),...,\mathcal{W}^{m-1}(\zeta'),\mathcal{W}(\zeta') )-\frac{\partial a^i}{\partial \eta_{p,j}^{k,l}}(\zeta, \mathcal{W}^0(\zeta'),...,\mathcal{W}^{m-1}(\zeta'),\mathcal{W}^m(\zeta') )\}dt\nonumber\\
&+&\int_0^1\{\frac{\partial a^i}{\partial \eta_{p,j}^{k,l}}(\zeta, \mathcal{W}^0(\zeta'),...,\mathcal{W}^{m-1}(\zeta'),\mathcal{W}^m(\zeta') )-\frac{\partial a^i}{\partial \eta_{p,j}^{k,l}}(\zeta, \mathcal{W}^0(\zeta), \mathcal{W}^1(\zeta'),...,\mathcal{W}^{m-1}(\zeta'),\mathcal{W}^m(\zeta') )\}dt\nonumber\\
&+&\cdot\cdot\cdot\nonumber\\
 &+&\int_0^1\{\frac{\partial a^i}{\partial \eta_{p,j}^{k,l}}(\zeta, \mathcal{W}^0(\zeta),...,\mathcal{W}^{m-1}(\zeta),\mathcal{W}^m(\zeta') )-\frac{\partial a^i}{\partial \eta_{p,j}^{k,l}}(\zeta, \mathcal{W}^0(\zeta), \mathcal{W}^1(\zeta),...,\mathcal{W}^{m-1}(\zeta),\mathcal{W}^m(\zeta) )\}dt,\nonumber
 \end{eqnarray}
 whence
\begin{eqnarray}
 &&|A^{p,j}_{k,l}(\zeta')-A^{p,j}_{k,l}(\zeta)|\nonumber\\
 &\leq&H_\alpha^A[R,\gamma]\bigg\{|\zeta'-\zeta|^\alpha+\sum_{l=0}^{m-1}\sum_{ i+j=l}
 (|\partial^i\bar \partial^j f(\zeta')-\partial^i\bar \partial^j f(\zeta)|+|\partial^i\bar \partial^j g(\zeta')-\partial^i\bar \partial^j g(\zeta)|)^\alpha\bigg\}\nonumber\\
 &+&H_1^B[R,\gamma]\sum_{i+j=m}
 \bigg\{|\partial^i\bar \partial^j f(\zeta')-\partial^i\bar \partial^j f(\zeta)|+|\partial^i\bar \partial^j g(\zeta')-\partial^i\bar \partial^j g(\zeta)|\bigg\}\nonumber\\
 &\leq&
 H_\alpha^A[R,\gamma]\bigg\{|\zeta'-\zeta|^\alpha+2\sum_{l=0}^{m-1}\sum_{ i+j=l}
 (12^mR^{m-l-1}\gamma)^\alpha|\zeta'-\zeta|^\alpha \bigg\}\nonumber\\
 &+&H_1^A[R,\gamma]\sum_{i+j=m}
 H_\alpha[\partial^i\bar \partial^j f]|\zeta'-\zeta|^\alpha+H_\alpha[\partial^i\bar \partial^j g]|\zeta'-\zeta|^\alpha.\nonumber
 \end{eqnarray}
Therefore, we have
\begin{eqnarray}
H_\alpha [A^{p,j}_{k,l}] &\leq&H_\alpha^A[R,\gamma]\bigg\{1+2\sum_{l=0}^{m-1}2^l
 (12^mR^{m-l-1}\gamma)^\alpha\bigg\}\nonumber\\
 &+&H_1^A[R,\gamma]\sum_{i+j=m}
 H_\alpha[\partial^i\bar \partial^j f]+H_\alpha[\partial^i\bar \partial^j g]
\end{eqnarray}
By (29), (30), we have
\begin{eqnarray}
\|A^{p,j}_{k,l}\|&=&|A^{p,j}_{k,l}|+(2R)^\alpha H_\alpha [A^{p,j}_{k,l}] \nonumber\\
&\leq& A(R,\gamma)+(2R)^\alpha H_\alpha^A[R,\gamma]\bigg\{1+2\sum_{l=0}^{m-1}2^l
 (12^mR^{m-l-1}\gamma)^\alpha\bigg\}\nonumber\\
 &+&H_1^A(R,\gamma)\sum_{i+j=m}
 (2R)^\alpha \{H_\alpha[\partial^i\bar \partial^j f]+H_\alpha[\partial^i\bar \partial^j g]\}\nonumber\\
 &\leq& A(R,\gamma)+(2R)^\alpha H_\alpha^A[R,\gamma]\bigg\{1+2\sum_{l=0}^{m-1}2^l
 (12^mR^{m-l-1}\gamma)^\alpha\bigg\}\nonumber\\
 &+& 2\gamma 2^m H_1^A(R,\gamma)
\end{eqnarray}
where we have used that $(2R)^\alpha H_\alpha[\partial^i\bar\partial^j]\leq \|f\|^{(m)}\leq \gamma$ if $i+j=m$.
Of course, $\|\bar A^{p,j}_{k,l}\|$ satisfies the same estimate. By the same argument, we can have
\begin{eqnarray}
\|B^{p,j}_{k,l}\|&\leq& B(R,\gamma)+(2R)^\alpha H_\alpha^B[R,\gamma]\bigg\{1+2\sum_{l=0}^{m-1}2^l
 (12^mR^{m-l-1}\gamma)^\alpha \bigg\}\nonumber\\
 &+& \gamma 2^{m+1} H_1^B(R,\gamma)
\end{eqnarray}
Putting (22),(23),(31), (32) all together we have thus shown that
\begin{eqnarray}
\|a^i(\zeta, \mathcal{D}^0f, \mathcal{D}^1 f, ...,\mathcal{D}^{m} f)-a^i(\zeta, \mathcal{D}^0g, \mathcal{D}^1 g, ...,\mathcal{D}^{m} g)\|&\leq& \delta_1(R,\gamma)\|f-g\|^{(m)}\\
\|\mathbf{\omega}^i(f)-\mathbf{\omega}^i(g)\|^{(m)}&\leq& M\delta_1(R,\gamma)|\|f-g\|^{(m)}
\end{eqnarray}
where $M$ is defined by (23), and
$$\delta_1(R,\gamma)=B(R,\gamma)+\gamma 2^{m+1} H_1^B(R,\gamma)+\delta_2(R,\gamma)$$
where
\begin{eqnarray}
&&\delta_2(R,\gamma)\nonumber\\
&=&\sum_{p=0}^{m-1}\frac{6^{m-p}2^{p+1}}{(m-p)!}R^{m-p}\bigg\{ A(R,\gamma)+\gamma2^{m+1}H_1^A[R,\gamma]+(2R)^\alpha H_\alpha^A[R,\gamma]\bigg\{1+2\sum_{l=0}^{m-1}2^l
 (12^mR^{m-l-1}\gamma)^\alpha\bigg\}\bigg\}\nonumber\\
&+&(2R)^\alpha H_\alpha^B[R,\gamma]\bigg\{1+2\sum_{l=0}^{m-1}2^l
 (12^mR^{m-l-1}\gamma)^\alpha\bigg\}.
\end{eqnarray}
Here we want to point out properties of $\delta_2(R,\gamma)$. For each given $\gamma>0$, we have $\lim_{R\to 0} \delta_2(R,\gamma)=0$.
Similarly, For each given $\gamma>0$, we have $\lim_{R\to 0} \delta_1(R,\gamma)=0$.
We write $M=2^{\frac{m(m-1)}{2}}(C_1m+C_0)^m+O(R^\alpha)$ where $O(R^\alpha)$ depends only on $m,\alpha$.
Now we have thus shown that
\begin{eqnarray}
\|\mathbf{\omega}^i(f)-\mathbf{\omega}^i(g)\|^{(m)}\leq\{2^{\frac{m(m-1)}{2}}(C_1m+C_0)^m(B(R,\gamma)+\gamma 2^{m+1} H_1^B[R,\gamma])+\delta_3(R,\gamma)\}\|f-g\|^{(m)}
\end{eqnarray}
where
\begin{eqnarray}
\delta_3(R,\gamma)=\delta_2(R,\gamma)2^{\frac{m(m-1)}{2}}(C_1m+C_0)^m+O(R^\alpha)(B(R,\gamma)+\gamma 2^{m+1} H_1^B[R,\gamma]+\delta_2(R,\gamma)).
\end{eqnarray}
Here we want to point out $\delta_3(R,\gamma)$ has the similar properties. In fact, For each given $\gamma>0$, we have $\lim_{R\to 0} \delta_3(R,\gamma)=0$.
\subsubsection{Estimate of $I_{10}$}
The following lemmas are needed before we can estimate $I_{10}$. First we assume $^0Tf=f, ^1Tf=Tf$.
\begin{lem}
Let $f\in C^\alpha(D)$, and let $\mu+\nu=m, k+l=m, k\not=\mu, l\not=\nu$. Then if $l>\nu$, it holds
$$\partial^k\bar\partial^l(\T^\nu\overline T^\mu f)=^{l-\nu+1}\overline T(\overline T^{\mu-k-1} f),$$
if $l<\nu$, it holds
$$\partial^k\bar\partial^l(\T^\nu\overline T^\mu f)=^{k+1}T(T^{\nu-l-1}(\overline T^\mu f)).$$
\end{lem}
\begin{proof}
If $l>\nu$, then $k<\mu$. Thus we have
\begin{eqnarray}
\partial^k\bar\partial^l(\T^\nu\overline T^\mu f)&=& \partial^k\bar\partial^{l-\nu}\bar\partial^\nu(T^\nu\overline T^\mu f)\nonumber\\
&=&\partial^k\bar\partial^{l-\nu}\overline T^\mu f\nonumber\\
&=&\bar\partial^{l-\nu}\partial^k\overline T^\mu f\nonumber\\
&=&\bar\partial^{l-\nu}\partial^k(\overline T^k\overline T^{\mu-k}f)\nonumber\\
&=&\bar\partial^{l-\nu}\overline T^{\mu-k} f\nonumber\\
&=&\bar\partial^{l-\nu}\overline T(\overline T^{\mu-k-1} f)\nonumber\\
&=&^{l-\nu+1}\overline T(\overline T^{\mu-k-1} f).\nonumber
\end{eqnarray}
The proof of the other case is similar. In fact,
\begin{eqnarray}
\partial^k\bar\partial^l(T^\nu\overline T^\mu f)&=&\partial^k\bar\partial^l(T^lT^{\nu-l}\overline T^\mu f)\nonumber\\
&=&\partial^k(T^{\nu-l}\overline T^\mu f)\nonumber\\
&=&\partial^kT(T^{\nu-l-1}\overline T^\mu f)\nonumber\\
&=&^{k+1}T(T^{\nu-l}\overline T^\mu f).\nonumber
\end{eqnarray}
\end{proof}
\begin{lem}
Let $h\in C^\alpha(D)$. If $k+l=m, \mu+\nu=m$, then
$$|\partial^k\bar\partial^l T^\nu\overline T^\mu(h)|\leq (C_1m+O(R^\alpha))\|h\|$$
where $O(R^\alpha)$ only depends on $m, \alpha$.
\end{lem}
\begin{proof}
We first note that if $l=\nu$, then $k=\mu$, and then $\partial^k\bar\partial^l T^\nu\overline T^\mu(h)=h$. Now if $l>\nu$, then by Lemma 4.4, noting $l-\nu=\mu-k$,
$$\partial^k\bar\partial^l(\T^\nu\overline T^\mu h)=^{l-\nu+1}\overline T(\overline T^{l-\nu-1} h).$$
Hence we have, by Lemma 3.9,
\begin{eqnarray}
|^{l-\nu+1}\overline T(\overline T^{l-\nu-1} h)|
\leq C_1mR^\alpha\sum_{i+j=l-\nu-1}H_\alpha[\partial^i\bar\partial^j\overline T^{l-\nu-1}h].\nonumber
\end{eqnarray}
Now we estimate, using Lemma 3.9, 4.4 repeatedly,
\begin{eqnarray}
&&\sum_{i+j=l-\nu-1}H_\alpha[\partial^i\bar\partial^j\overline T^{l-\nu-1}h]\nonumber\\
&=&\sum_{0\leq i\leq l-\nu-1}H_\alpha[\bar\partial^{l-\nu-1-i}\partial^i\overline T^{l-\nu-1}h]\nonumber\\
&=&\sum_{0\leq i< l-\nu-1}H_\alpha[\bar\partial^{l-\nu-1-i}\partial^i\overline T^{l-\nu-1}h]+H_\alpha[h]\nonumber\\
&=&\sum_{0\leq i< l-\nu-1}H_\alpha[\bar\partial^{l-\nu-1-i}\overline T^{l-\nu-1-i}h]+H_\alpha[h]\nonumber\\
&=&\sum_{0\leq i< l-\nu-1}H_\alpha[^{l-\nu-i}\overline T(\overline T^{l-\nu-2-i}h)]+H_\alpha[h]\nonumber\\
&\leq& C_1mR^\alpha \sum_{0\leq i< l-\nu-1}\sum_{p+q=l-\nu-2-i}H_\alpha[\partial^p\bar\partial^q(\overline T^{l-\nu-2-i}h)]+H_\alpha[h]
\end{eqnarray}
Therefore it holds
$$|^{l-\nu+1}\overline T(\overline T^{l-\nu-1} h)|\leq C_1mR^\alpha H_\alpha[h]+(C_1mR^\alpha)^2\sum_{0\leq i< l-\nu-1}\sum_{p+q=l-\nu-2-i}H_\alpha[\partial^p\bar\partial^q(\overline T^{l-\nu-2-i}h)]$$
Now we can apply equation (38) to above repeatedly until we get term $H_\alpha[h]$. This proves the lemma in this case.
Now we estimate the case $l<\nu$. This case is more complicated.  We have $k>\mu, $ and $\nu-l=k-\mu.$ By Lemma 4.4, we have
$$\partial^k\bar\partial^l(\T^\nu\overline T^\mu h)=^{k+1}T(T^{\nu-l-1}(\overline T^\mu h))=^{k+1}T(T^{k-\mu-1}(\overline T^\mu h)).$$
We have then, using Lemma 3.9
$$|^{k+1}T(T^{k-\mu-1}(\overline T^\mu h))|\leq
C_1mR^\alpha\sum_{i+j=k-1}H_\alpha[\partial^i\bar\partial^j(T^{k-\mu-1}(\overline T^\mu h))].$$
On the other hand, we have
\begin{eqnarray}
&&\sum_{i+j=k-1}H_\alpha[\partial^i\bar\partial^j(T^{k-\mu-1}(\overline T^\mu h))]\nonumber\\
&=&\sum_{0\leq j\leq k-1}H_\alpha[\partial^{k-1-j}\bar\partial^j(T^{k-\mu-1}(\overline T^\mu h))]\nonumber\\
&=&\sum_{0\leq j<k-\mu-1}H_\alpha[\partial^{k-1-j}\bar\partial^j(T^{k-\mu-1}(\overline T^\mu h))]+H_\alpha[h]
+\sum_{k-\mu-1< j\leq k-1}H_\alpha[\partial^{k-1-j}\bar\partial^j(T^{k-\mu-1}(\overline T^\mu h))]\nonumber\\
&=&\sum_{0\leq j<k-\mu-1}H_\alpha[\partial^{k-1-j}(T^{k-\mu-1-j}(\overline T^\mu h))]+H_\alpha[h]
+\sum_{k-\mu-1< j\leq k-1}H_\alpha[\partial^{k-1-j}\bar\partial^{j-k+\mu+1}((\overline T^\mu h)]\nonumber\\
&=&\sum_{0\leq j<k-\mu-1}H_\alpha[\partial^{k-1-j}T(T^{k-\mu-2-j}(\overline T^\mu h))]+H_\alpha[h]
+\sum_{k-\mu-1< j\leq k-1}H_\alpha[\bar\partial^{j-k+\mu+1}(\overline T^{\mu-k+1+j} h)]\nonumber\\
&=&\sum_{0\leq j<k-\mu-1}H_\alpha[^{k-j}T(T^{k-j-2-\mu}(\overline T^\mu h))]+H_\alpha[h]
+\sum_{k-\mu-1< j\leq k-1}H_\alpha[^{j-k+\mu+2}\overline T(\overline T^{\mu-k+j} h)]\nonumber
\end{eqnarray}
Using Lemma 3.9, we have
\begin{eqnarray}
&&\sum_{i+j=k-1}H_\alpha[\partial^i\bar\partial^j(T^{k-\mu-1}(\overline T^\mu h))]\nonumber\\
&\leq& C_1mR^\alpha\bigg\{\sum_{0\leq j<k-\mu-1}\sum_{p+q=k-j-2}H_\alpha[\partial^p\bar\partial^q(T^{k-\mu-2-j}(\overline T^\mu h))]\nonumber\\
&+&\sum_{k-\alpha-1< j\leq k-1}\sum_{p+q=\mu-k+j}H_\alpha[\partial^p\bar\partial^q(\overline T^{\mu-k+j} h)]\bigg\}+H_\alpha[h].\nonumber
\end{eqnarray}
Hence we have
\begin{eqnarray}
|^{k+1}T(T^{k-\mu-1}(\overline T^\mu h))|
&\leq& C_1mR^\alpha H_\alpha[h]\nonumber\\
&+&(C_1mR^\alpha)^2\bigg\{\sum_{0\leq j<k-\mu-1}\sum_{p+q=k-j-2}H_\alpha[\partial^p\bar\partial^q(T^{k-\mu-2-j}(\overline T^\mu h))]\nonumber\\
&+&\sum_{k-\mu-1< j\leq k-1}\sum_{p+q=\mu-k+j}H_\alpha[\partial^p\bar\partial^q(\overline T^{\mu-k+j} h)]\bigg\}.
\end{eqnarray}
In order to finish off the proof, we need to use equations (38), (39) repeatedly to get the desired result.
\end{proof}
Now we are ready to estimate $I_{10}$. Indeed,
\begin{eqnarray}
I_{10}&=&\bigg|\sum_{k+l=m}\{[\partial^k\bar\partial^l \omega^i(f)](0)
-[\partial^k\bar\partial^l \omega^i(g)](0)\}\bigg|\nonumber\\
&\leq&\sum_{k+l=m}\bigg|[\partial^k\bar\partial^l \omega^i(f)](0)
-[\partial^k\bar\partial^l \omega^i(g)](0)\bigg|\nonumber\\
&=&\sum_{k+l=m}\bigg|[\partial^k\bar\partial^l T^\nu\overline T^\mu (a^i(f)-a^i(g))](0)\bigg|\nonumber\\
&\leq&\sum_{k+l=m}\bigg|\partial^k\bar\partial^l T^\nu\overline T^\mu (a^i(f)-a^i(g))\bigg|\nonumber\\
&\leq& 2^m(C_1m+O(R^\alpha))\|a^i(f)-a^i(g)\|\nonumber\\
&\leq&2^m(C_1m+O(R^\alpha))\delta_1(R,\gamma)\|f-g\|^{(m)}\nonumber\\
&\leq& (m2^mC_1(B(R,\gamma)+\gamma 2^{m+1} H_1^B[R,\gamma])+\delta_4(R,\gamma))\|f-g\|^{(m)}
\end{eqnarray}
where
$$\delta_4(R,\gamma)=\delta_2(R,\gamma)2^mC_1m+O(R^\alpha)(B(R,\gamma)+\gamma 2^{m+1} H_1^B[R,\gamma]+\delta_2(R,\gamma))$$
and where we denote $a^i(f)=a^i(z, f, \mathcal{D}^1 f, ...,\mathcal{D}^{m-1} f, \mathcal{D}^m f)$.
Once again we have $\lim_{R\to0}\delta_4(R,\gamma)=0$ for each $\gamma>0$.

\subsubsection{Estimate of $I_{11}$}
In order to estimate $I_{11}$, we begin with
\begin{eqnarray}
&&a^i(\zeta, \mathcal{D}^0f(\zeta), \mathcal{D}^1 f(\zeta), ...,\mathcal{D}^{m-1} f(\zeta), \mathcal{D}^m f(\zeta))-a^i(0,...,0)\nonumber\\
&=&\int_0^1\frac{d}{dt}a^i(t\zeta, tf(\zeta),..., t\mathcal{D}^{m-1} f(\zeta), t\mathcal{D}^m f(\zeta))dt\nonumber\\
&=&C^i\zeta+\bar C^i\bar\zeta+\sum_{j=0}^n\sum_{p=0}^{m-1}\sum_{k+l=p}A^{p,j}_{k,l}\partial^k\bar\partial^l(f_j)+\bar A^{p,j}_{k,l}\overline{\partial^k\bar\partial^l(f_j)}\nonumber\\
&+&\sum_{j=0}^n\sum_{k+l=m}B^{m,j}_{k,l}\partial^k\bar\partial^l(f_j)+\bar B^{m,j}_{k,l}\overline{\partial^k\bar\partial^l(f_j)}
\end{eqnarray}
where we have
$$C^i=\int_0^1\frac{\partial a^i}{\partial\zeta}(t\zeta, tf(\zeta),..., t\mathcal{D}^{m-1} f(\zeta), t\mathcal{D}^m f(\zeta))dt,$$
$$\bar C^i=\int_0^1\frac{\partial a^i}{\partial\bar\zeta}(t\zeta, tf(\zeta),..., t\mathcal{D}^{m-1} f(\zeta), t\mathcal{D}^m f(\zeta))dt,$$
$$A^{p,j}_{k,l}=\int_0^1\frac{\partial a^i}{\partial \eta_{p,j}^{k,l}}(t\zeta, tf(\zeta),..., t\mathcal{D}^{m-1} f(\zeta), t\mathcal{D}^m f(\zeta))dt,$$
$$\bar A^{p,j}_{k,l}=\int_0^1\frac{\partial a^i}{\bar\partial \eta_{p,j}^{k,l}}(t\zeta, tf(\zeta),..., t\mathcal{D}^{m-1} f(\zeta), t\mathcal{D}^m f(\zeta))dt,$$
$$B^{m,j}_{k,l}=\int_0^1\frac{\partial a^i}{\partial \eta_{m,j}^{k,l}}(t\zeta, tf(\zeta),..., t\mathcal{D}^{m-1} f(\zeta), t\mathcal{D}^m f(\zeta))dt,$$
$$\bar B^{m,j}_{k,l}=\int_0^1\frac{\partial a^i}{\bar\partial \eta_{m,j}^{k,l}}(t\zeta, tf(\zeta),..., t\mathcal{D}^{m-1} f(\zeta), t\mathcal{D}^m f(\zeta))dt.$$
Here we notice that these functions are slightly different from those in (24) in arguments but similar, but we abuse the notations anyway, hoping causing no confusion.
Taking norms on (41), we have

\begin{eqnarray}
&&\|a^i(\zeta, \mathcal{D}^0f(\zeta), \mathcal{D}^1 f(\zeta), ...,\mathcal{D}^{m-1} f(\zeta), \mathcal{D}^m f(\zeta))\|\leq |a^i(0,...,0)|\nonumber\\
&+&\|C^i\|\|\zeta\|+\|\bar C^i\|\|\bar\zeta\|+\sum_{j=0}^n\sum_{p=0}^{m-1}\|\sum_{k+l=p}A^{p,j}_{k,l}\|\|\partial^k\bar\partial^l(f_j)\|+\|\bar A^{p,j}_{k,l}\|\|\overline{\partial^k\bar\partial^l(f_j)}\|\nonumber\\
&+&\sum_{j=0}^n\sum_{k+l=m}\|B^{m,j}_{k,l}\|\|\partial^k\bar\partial^l(f_j)\|+\|\bar B^{m,j}_{k,l}\|\|\overline{\partial^k\bar\partial^l(f_j)}\|+\nonumber\\
&\leq&3R(\|C^i\|+\|\bar C^i\|)+\sum_{p=0}^{m-1}\sum_{k+l=p}\{\|A^{p,j}_{k,l}\|+\|\bar A^{p,j}_{k,l}\|\}\|{\partial^k\bar\partial^l f}\|\nonumber\\
&+&\sum_{k+l=m}\{\|B^{m,j}_{k,l}\|+\|\bar B^{m,j}_{k,l}\|\}\|{\partial^k\bar\partial^l f}\|+|a^i(0,...,0)|\nonumber\\
&\leq&3R(\|C^i\|+\|\bar C^i\|)+\sum_{p=0}^{m-1}\sum_{k+l=p}\{\|A^{p,j}_{k,l}\|+\|\bar A^{p,j}_{k,l}\|\}\frac{6^{m-p}}{(m-p)!}R^{m-p}\|f\|^{(m)}\nonumber\\
&+&\sum_{k+l=m}\{\|B^{m,j}_{k,l}\|+\|\bar B^{m,j}_{k,l}\|\}\|f\|^{(m)}+|a^i(0,...,0)|\nonumber\\
&\leq& |a^i(0,...,0)|+3R(\|C^i\|+\|\bar C^i\|)+\bigg\{\sum_{p=0}^{m-1}\frac{6^{m-p}}{(m-p)!}R^{m-p}\sum_{k+l=p}\{\|A^{p,j}_{k,l}\|+\|\bar A^{p,j}_{k,l}\|\}\nonumber\\
&+&\sum_{k+l=m}\{\|B^{m,j}_{k,l}\|+\|\bar B^{m,j}_{k,l}\|\}\bigg\}\|f\|^{(m)}.
\end{eqnarray}
Using the identical arguments as for (30), (31),(32), we have
\begin{eqnarray}
\|C^i\|&\leq& C(R,\gamma)+(2R)^\alpha H_\alpha^C[R,\gamma]\bigg\{1+\sum_{l=0}^{m-1}2^l
 (12^m R^{m-1-1}\gamma)^\alpha \bigg\}
 + \gamma 2^m H_1^C(R,\gamma)\nonumber\\
 \|A^{p,j}_{k,l}\|&\leq& A(R,\gamma)+(2R)^\alpha H_\alpha^A[R,\gamma]\bigg\{1+\sum_{l=0}^{m-1}2^l
 (12^m R^{m-1-1}\gamma)^\alpha \bigg\}
 + \gamma 2^m H_1^A(R,\gamma)\nonumber\\
 \|B^{p,j}_{k,l}\|&\leq& B(R,\gamma)+(2R)^\alpha H_\alpha^B[R,\gamma]\bigg\{1+\sum_{l=0}^{m-1}2^l
 (12^m R^{m-1-1}\gamma)^\alpha \bigg\}
 + \gamma 2^m H_1^B(R,\gamma).
\end{eqnarray}
Of course, $\|\bar C^i\|,\|\bar A^{p,j}_{k,l}\|, \|\bar B^{p,j}_{k,l}\|$ will satisfy the same estimate respectively. Similar to (33), we have
$$\|a^i(f)\|\leq |a^i(0)|+\delta_5(R,\gamma)$$
where $\delta_5(R,\gamma)$ is defined as
$$\delta_5(R,\gamma)=6R\bigg(C(R,\gamma)+(2R)^\alpha H_\alpha^C[R,\gamma]\{1+\sum_{l=0}^{m-1}2^l
 (12^m R^{m-1-1}\gamma)^\alpha\}+\gamma 2^m H_1^C[R,\gamma]\bigg)+$$
 $$\gamma\bigg(\sum_{p=0}^{m-1}\frac{6^{m-p}}{(m-p)!}R^{m-p} 2^p( A(R,\gamma)+(2R)^\alpha H_\alpha^A[R,\gamma]\{1+\sum_{l=0}^{m-1}2^l
 (12^m R^{m-1-1}\gamma)^\alpha\}+\gamma 2^m H_1^A[R,\gamma])+$$
 $$B(R,\gamma)+(2R)^\alpha H_\alpha^B[R,\gamma]\{1+\sum_{l=0}^{m-1}2^l
 (12^m R^{m-1-1}\gamma)^\alpha\}+ \gamma 2^m H_1^B[R,\gamma]\bigg).$$
 However, we will replace $\delta_5(R,\gamma)$ by
 $$\delta_5(R,\gamma)=6R\bigg(C(R,\gamma)+(2R)^\alpha H_\alpha^C[R,\gamma]\{1+\sum_{l=0}^{m-1}2^l
 (12^m R^{m-1-1}\gamma)^\alpha\}+\gamma 2^m H_1^C[R,\gamma]\bigg)+\gamma\delta_1(R,\gamma).$$
 We remark that $\lim_{R\to0}\delta_5(R,\gamma)=0$.
It follows
\begin{eqnarray}
\|\omega^i(f)\|^{(m)}&\leq& M\|a^i(f)\|\nonumber\\
&\leq& M(|a^i(0)|+\delta_5(R,\gamma))\nonumber\\
&\leq& (2^{\frac{m(m-1)}{2}}(C_1m+C_0)^m+O(R^\alpha))(|a^i(0)|+\delta_5(R,\gamma))
\end{eqnarray}
where we denote $a^i(f)=a^i(z, f, \mathcal{D}^1 f, ...,\mathcal{D}^{m-1} f, \mathcal{D}^m f)$.

\subsubsection{Estimate of $I_{12}$}

\begin{eqnarray}
I_{12}&=&\bigg|\sum_{k+l=m}[\partial^k\bar\partial^l\omega^i(f)](0)\bigg|\nonumber\\
&\leq &\sum_{k+l=m}\bigg|[\partial^k\bar\partial^l\omega^i(f)](0)\bigg|\nonumber\\
&=&\sum_{k+l=m}\bigg|[\partial^k\bar\partial^l T^\nu\overline T^\mu (a^i(f))](0)\bigg|\nonumber\\
&\leq&\sum_{k+l=m}\bigg|\partial^k\bar\partial^l T^\nu\overline T^\mu (a^i(f))\bigg|\nonumber\\
&\leq& 2^m(C_1m+O(R^\alpha))\|a^i(f)\|\nonumber\\
&\leq&2^m(C_1m+O(R^\alpha))(|a^i(0)|+\delta_5(R,\gamma))
\end{eqnarray}

\subsection{A general estimate}
Here we collect the estimates together for later quick reference.
\begin{thm}
Let $\mathbf{\Theta}:\mathbf{B}(R)\to \mathbf{B}(R)$ be defined as in (19) and (20). Then if $f,g\in\mathbf{A}(R,\gamma)$ then
\begin{eqnarray}
\|\mathbf{\Theta}(f)-\mathbf{\Theta}(g)\|^{(m)}
&\leq &\delta(R,\gamma)\|f-g\|^{(m)}\\
\|\mathbf{\Theta}(f)\|^{(m)}&\leq& \eta(R,\gamma)
\end{eqnarray}
where
\begin{eqnarray}
\delta(R,\gamma)&=&2^{\frac{m(m-1)}{2}}(C_1m+C_0)^m(B(R,\gamma)+\gamma 2^{m+1} H_1^B(R,\gamma))+\delta_3(R,\gamma)\nonumber\\
&+&m2^mC_1(B(R,\gamma)+\gamma 2^{m+1} H_1^B(R,\gamma))+\delta_4(R,\gamma)\\
\eta(R,\gamma)&=&2^{\frac{m(m-1)}{2}}(C_1m+C_0)^m+O(R^\alpha))(|a(0)|+\delta_5(R,\gamma))\nonumber\\
&+&2^m(C_1m+O(R^\alpha))(|a(0)|+\delta_5(R,\gamma)).
\end{eqnarray}
Furthermore, $\lim_{R\to0}\delta_j(R,\gamma)=0$ for each given $\gamma>0$.
\end{thm}
For easy references, we collect the definitions of $\delta_j(R,\gamma)$ for $i=1,2,3,4,5$ here.
\begin{eqnarray}
\delta_1(R,\gamma)&=&B(R,\gamma)+\gamma 2^{m+1} H_1^B(R,\gamma)+\delta_2(R,\gamma)\\
\delta_2(R,\gamma)&=&\sum_{p=0}^{m-1}\frac{6^{m-p}2^{p+1}}{(m-p)!}R^{m-p}\bigg\{ A(R,\gamma)+\gamma2^{m+1}H_1^A[R,\gamma]+(2R)^\alpha H_\alpha^A[R,\gamma]\{1+\nonumber\\
&&2\sum_{l=0}^{m-1}2^l
 (12^mR^{m-l-1}\gamma)^\alpha\}\bigg\}
+(2R)^\alpha H_\alpha^B[R,\gamma]\{1+2\sum_{l=0}^{m-1}2^l
 (12^mR^{m-l-1}\gamma)^\alpha\}\\
 \delta_3(R,\gamma)&=&\delta_2(R,\gamma)2^{\frac{m(m-1)}{2}}(C_1m+C_0)^m\nonumber\\
 &+&O(R^\alpha)(B(R,\gamma)+\gamma 2^{m+1} H_1^B[R,\gamma]+\delta_2(R,\gamma))\\
 \delta_4(R,\gamma)&=&\delta_2(R,\gamma)2^mC_1m+O(R^\alpha)(B(R,\gamma)+\gamma 2^{m+1} H_1^B[R,\gamma]+\delta_2(R,\gamma))\\
 \delta_5(R,\gamma)&=&6R\bigg(C(R,\gamma)+(2R)^\alpha H_\alpha^C[R,\gamma]\{1+\sum_{l=0}^{m-1}2^l
 (12^m R^{m-1-1}\gamma)^\alpha\}\nonumber\\
 &+&\gamma 2^m H_1^C[R,\gamma]\bigg)+\gamma\delta_1(R,\gamma).
\end{eqnarray}

\section{Proofs of Theorems}
A few words on constants are in order. So far we have used constants $C_0, C_1, C_2$ in previous sections and will continue to use $C_3, C_4,...$ for constants dependent on parameters and they may vary from line to line. We also use $C(\cdot,...,\cdot)$ to emphasize the dependence of the parameters specifically, and they may vary from line to line.
\subsection{Proof of Theorem 1.2}
To start the proof of the theorem, we fix an $\alpha$ ($o<\alpha<1)$ so that $a\in C^{k-1,\alpha}$. Then we can apply the general estimates (46), (47).
First we rewrite $\delta(R,\gamma)$ in (48) as follows
$$\delta(R,\gamma)=C_3B(R,\gamma)+C_4\gamma H_1^B(R,\gamma)+\varepsilon(R,\gamma)$$
where $C_3, C_4$ are constants only depend on $m,\alpha$ and
$\varepsilon(R,\gamma)$ is such that $\lim_{R\to0} \varepsilon(R,\gamma)=0$ for each $\gamma>0$. This notation will be used again but may vary from line to line.
Second, we rewrite $\eta(R,\gamma)$ as in (49)
$$\eta(R,\gamma)=C(m,R,\alpha)|a(0)|+C_5(\gamma B(R,\gamma)+\gamma^22^{m+1} H_1^B[R,\gamma])+\varepsilon(R,\gamma)$$
where
$$C_5=2^{\frac{m(m-1)}{2}}(C_1m+C_0)^m+2^mC_1m,$$
and $C(m,R,\alpha)\to C_5$ as $R\to 0$.
We now give estimates of $B(R,\gamma), H_1^B(R,\gamma)$ in terms of conditions (1),(4).
We assume that $\xi$ is one of the coordinates of $\eta=\eta_m\in \co^{n2^m}$. Let $\sigma=\eta_{m,j}^{kl}$ be a component variable of $\eta_m$.
We want to estimate the Lipschitz constant of $\partial_{\sigma} a$.
In fact, we have
$$\partial_{\sigma} a(\cdot, \xi)-\partial_{\sigma} a(\cdot, \xi')=\int_{0}^1 \frac{d}{dt}\partial_{\sigma} a(\cdot, t\xi+(1-t)\xi')dt$$
$$=\int_0^1 \partial_{\sigma}\partial_\xi a(\cdot,t\xi+(1-t)\xi')(\xi-\xi')+\partial_{\sigma}\bar\partial_\xi a(\cdot,t\xi+(1-t)\xi')\overline{(\xi-\xi')}dt$$
Here $a(\cdot,\xi)$ is defined naturally for what $\xi$ is.
It follows that
$$H_1[\partial_{\sigma} a]|_{E(R,\gamma)}\leq  |\partial_{\sigma}\partial_\xi a|_{E(R,\gamma)}+|\partial_{\sigma}\bar\partial_\xi a|_{E(R,\gamma)}$$
$$\leq |\partial_{\sigma}\partial_\xi a(0)|+|\partial_{\sigma}\bar\partial_\xi a(0)|+o(R+\gamma)$$
where $o(R+\gamma)\to 0$ as $R,\gamma\to 0$ by continuity because of $C^2$ smoothness of $a$ and $\lim_{R,\gamma\to0} E(R,\gamma)=\{0\}$.
The estimate for $\bar\partial_{\sigma}$ is the same. So we have
$$H_1^B[R,\gamma]\leq |\partial_\eta\partial_\eta a(0)|+|\partial_\eta\bar\partial_\eta a(0)|+|\bar\partial_\eta\bar\partial_\eta a(0)|+o(R+\gamma).$$
On the other hand, we have
\begin{eqnarray}
&&\partial_\sigma a^i(\zeta,\eta_1,...,\eta_{m-1},\eta_m)-\partial_\sigma a(0)\nonumber\\
&=&\int_0^1 \frac{d}{dt}\partial_\sigma a^i(t\zeta, t\eta_1,...,t\eta_{m-1},t\eta_m)dt
=\int_0^1 \{\partial_\sigma\partial_\zeta a^i(\cdot)\zeta+\partial_\sigma\bar\partial_\zeta a(\cdot)\bar\zeta\nonumber\\
&+&\sum_{j=1}^n\sum_{p=0}^{m-1}\sum_{k+l=p}\partial_\sigma\partial_{\eta^{kl}_{p,j}} a^i(\cdot)\eta^{kl}_{p,j}+\partial_\sigma\bar\partial_{\eta^{kl}_{p,j}} a(\cdot)\bar\eta^{kl}_{p,j}\nonumber\\
&+&\sum_{j=1}^n\sum_{k+l=m}\partial_\sigma\partial_{\eta^{kl}_{m,j}} a^i(\cdot)\eta^{kl}_{m,j}+\partial_\sigma\bar\partial_{\eta^{kl}_{m,j}} a^i(\cdot)\bar\eta^{kl}_{m,j}\}dt.
\end{eqnarray}
Notice that for $(\zeta,\eta_1,...,\eta_{m-1},\eta_m)\in E(R,\gamma)$, we have $|\zeta|\leq R, |\eta_j|\leq O(R^{m-j}\gamma), j=1,...,m-1$ and $|\eta_m|\leq \gamma$.
Hence, we have by (55), putting terms with $R$ and $\gamma$ together
$$|\partial_\sigma a^i|_{E(R,\gamma)}\leq |\partial_\sigma a(0)|+ 2^m(|\partial_\eta\partial_{\eta} a(0)|+|\partial_\eta\bar\partial_{\eta} a(0)|+
|\bar\partial_\eta\bar\partial_{\eta} a(0)|)\gamma+\varepsilon(R,\gamma)$$
where $\varepsilon(R,\gamma)\to 0$ as $R\to 0$ for each given $\gamma$. This implies
$$B(R,\gamma)\leq |\partial_\sigma a(0)|+|\bar\partial_\sigma a(0)|+n2^m(|\partial_\eta\partial_{\eta} a(0)|+|\partial_\eta\bar\partial_{\eta} a(0)|+
|\bar\partial_\eta\bar\partial_{\eta} a(0)|)\gamma+\varepsilon(R,\gamma).$$
Let
$$\tau=|\partial_\sigma a(0)|+|\bar\partial_\sigma a(0)|$$
$$\kappa=|\partial_\eta\partial_{\eta} a(0)|+|\partial_\eta\bar\partial_{\eta} a(0)|+
|\bar\partial_\eta\bar\partial_{\eta} a(0)|.$$
we have
$$\delta(R,\gamma)
\leq C_3\tau+(n2^mC_3\kappa+C_4\kappa+o(R+\gamma))\gamma+\varepsilon(R,\gamma),$$
$$\eta(R,\gamma)\leq C(m,R,\alpha)|a(0)|+C_5\left(n2^{m}\kappa+o(R+\gamma)\right)\gamma+\varepsilon(R,\gamma).$$
Now we fix $\gamma$. By $o(R+\gamma)\to 0$, there exist $R_0, \gamma_0>0$ such that
$$o(R+\gamma_0)\leq \min\{n2^{m}\kappa,n2^mC_3\kappa+C_4\kappa\},\mbox{ and } C(m,R,\alpha)<2C_5$$
for $R\leq R_0$.
Hence it holds for $f, g\in \mathbf{A}(R,\gamma_0)$
$$\delta(R,\gamma_0)
\leq C_3 \tau2(n2^mC_3\kappa+C_4\kappa)\gamma_0+\varepsilon(R,\gamma_0),$$
$$\eta(R,\gamma_0)\leq C_52n2^{m}\kappa\gamma_0+\varepsilon(R,\gamma_0).$$
Let $\delta=(C_52^{m+3})^{-1}$. If $\kappa<\delta$, then
$$\delta(R,\gamma_0)
\leq 2C_5|a(0)|+\{2(n2^mC_3\delta+C_4\delta)\gamma_0+\varepsilon(R,\gamma_0)\},$$
$$\eta(R,\gamma_0)\leq \frac{\gamma_0}{4}+\varepsilon(R,\gamma_0).$$
Now further choose $\gamma_0$ so that
$$2(n2^mC_3\delta+C_4\delta)\gamma_0\leq \frac{1}{2}$$
and choose $R$ small enough that $\varepsilon(R,\gamma_0)\leq \min\{1/4,\gamma_0/8\}$. Finally we have $\delta(R,\gamma_0)\leq C_3\tau+ 3/4,\eta(R,\gamma_0)
\leq 2C_5|a(0)|+3\gamma_0/8$, and finally choose $C_3\tau<1/8, 2C_5|a(0)|<\gamma_0/8$.
This is equivalent to that for $f,g\in \mathbf{A}(R,\gamma_0)$
$$\|\mathbf{\Theta}(f)-\mathbf{\Theta}(f)\|^{(m)}\leq \frac{7}{8}\|f-g\|^{(m)},$$
$$\|\mathbf{\Theta}(f)\|^{(m)}\leq \frac{\gamma_0}{2}.$$
For the $\delta$ we can now take $\delta=\min\{(C_52^{m+3})^{-1},\frac{1}{8C_3}\}$ and $\epsilon=\frac{\gamma_0}{16C_5}$.
Now let $\psi(z)$ be a map to $\co^n$ whose component is a homogenous polynomial of degree $m$ without term $z^\mu\bar z^\nu$. And further we assume
$\|\psi\|^{(m)}\leq \gamma_0/2$. Then by Lemma 4.1, the following (integral) equation
$$u=\psi+\mathbf{\Theta}(u)$$
has solutions in $A(R,\gamma_0)$ such that $\partial^i\bar\partial^j u(0)= \partial^i\bar\partial^j\psi(0)$ for $i+j=m, i\not=\mu,j\not=\nu$.
This is equivalent to equation (3). It is obvious that the solution is of vanishing order $m$ at the origin. The proof is complete.
\subsection{Proof of Theorem 1.3}
First we consider the case $p^i(z)=0$ for $i=1,...,n$.
 Since $a$ is independent of $\eta_m$, we have $B(R,\gamma)=H_1^B[R,\gamma]=H_\alpha^B[R,\gamma]=0$. Here we note that $C^{k,\alpha} (k\geq 1, 0<\alpha<1)$ regularity of $a$ is only needed. Therefore, we can replace
the constants as follows
$$\delta(R,\gamma)=\delta_3(R,\gamma)+\delta_4(R,\gamma),$$
$$\delta_4(R,\gamma)=C(R,\gamma)\delta_2(R,\gamma)$$
where $\lim_{R\to0}\delta_j(R,\gamma)=0$ for each $\gamma$ for $j=2,3,4$.
Since $O(R^\alpha)\to 0$ as $R\to0$, we can choose $R$ so small that we can replace $\eta(R,\gamma)$ by
$$\eta(R,\gamma)=C(m,\alpha)(|a(0)|+\delta_5(R,\gamma))$$
On the other hand, we can write
$$\delta_5(R,\gamma_0)=\epsilon(R,\gamma)+\gamma\delta_1(R,\gamma)$$
where $\lim_{R\to0}\epsilon(R,\gamma)=0$ for each $\gamma$. Therefore we have
$$\eta(R,\gamma)=C(m,\alpha)(|a(0)|+\epsilon(R,\gamma)+\gamma\delta_1(R,\gamma)).$$
Now we choose $\gamma_0$ so large that $\frac{\gamma_0}{8}>C(m,\alpha)|a(0)|$. Then we choose $R$ sufficiently small that
$C(m,\alpha)\epsilon(R,\gamma_0)\leq \frac{\gamma_0}{8}$ and $C(m,\alpha)\delta_1(R,\gamma_0)\leq \frac{\gamma_0}{4}$.
As a result, we have
$$\eta(R,\gamma_0)<\frac{\gamma_0}{2}.$$
If necessary, we choose $R$ further so that
$$\delta(R,\gamma_0)<3/4.$$
Let $\psi(z)$ be a homogenous polynomial map of degree $m$ without term $\partial^\mu\bar\partial^\nu$ such that $\|\psi\|\leq \frac{\gamma_0}{2}$. Then
we can apply Lemma 4.1 to
$$u=\psi+\mathbf{\Theta}(u)$$
on $A(R,\gamma_0)$ to get the desired solutions that vanish up to order $m-1$ at the origin.

To get the general case, we consider system, $p(z)=(p^1,...,p^n),$
$$\partial^\mu\bar\partial^\nu \tilde u=a(z, \tilde u+p(z), \mathcal{D}^1\tilde u+\mathcal{D}^1p(z),..., \mathcal{D}^{m-1}\tilde u+\mathcal{D}^{m-1}p(z)).$$
This can be written as
$$\partial^\mu\bar\partial^\nu \tilde u=b(z, \tilde u, \mathcal{D}^1\tilde u,..., \mathcal{D}^{m-1}\tilde u).$$
As long as $p(0)\in \mathrm{Int}(D')$, we can solve $\tilde u$ as just proved so that $\tilde u$ vanishes up to order $m-1$ at the origin.
Then $u=\tilde u+p(z)$ solve the original equation with desired property of derivatives at the origin since $\partial^\mu\bar\partial^\nu \tilde u=
\partial^\mu\bar\partial^\nu u$. The proof is complete.

\subsection{Proof of Theorem 1.4}
Here we fix $R$, and $\gamma$ will be chosen very small to prove the existence of global solutions. Since $a$ is independent of $z$, so we have
$C(R,\gamma)=H_\alpha^C[R,\gamma]=H_1^C[R,\gamma]=0$. Therefore we can substitute $\eta(R,\gamma),\delta(R,\gamma)$ by
\begin{eqnarray}
\eta(R,\gamma)&=&C_6(|a(0)|+\gamma\delta_1(R,\gamma))\\
\delta(R,\gamma)&=&C_7B(R,\gamma)+C_8\gamma H_1^B[R,\gamma]+\delta_3(R,\gamma)+\delta_4(R,\gamma)
\end{eqnarray}
where $C_6, C_7, C_8$ are constants dependent on only $m,\alpha$, and $R$.

Since $a$ is independent of $z$ variable, $E(R,\gamma)$ can be taken as
$$E(R, \gamma)=\Pi_{k=0}^{m-1}\{z\in\co^{n2^k}||z|\leq 6^mR^{m-k}\gamma\}\times\{z\in \co^{n2^m}||z|\leq\gamma\}.$$
We notice $E(R, \gamma)$ shrinks to the origin as $\gamma\to0$ for a fixed $R$. This is important for what follows in proving the theorem.
Let $\sigma$ be one of component variables in $\{\eta_0,...,\eta_m\}$. We have
\begin{eqnarray}
&&\partial_\sigma a^i(\eta_1,...,\eta_{m-1},\eta_m)=\partial_\sigma a^i(0)+\int_0^1 \frac{d}{dt}\partial_\sigma a^i(t\eta_1,...,t\eta_{m-1},t\eta_m)dt\nonumber\\
&=&\partial_\sigma a^i(0)+\int_0^1\{\sum_{j=1}^n\sum_{p=0}^{m}\sum_{k+l=p}\partial_\sigma\partial_{\eta^{kl}_{p,j}} a^i(\cdot)\eta^{kl}_{p,j}+\partial_\sigma\bar\partial_{\eta^{kl}_{p,j}} a(\cdot)\bar\eta^{kl}_{p,j}\}dt
\end{eqnarray}
Notice that for $(\eta_1,...,\eta_{m-1},\eta_m)\in E(R,\gamma)$, we have $ |\eta^{kl}_{p,j}|\leq 6^mR^{m-p}\gamma,k+l=p; p=1,...,m-1,j=1,...,n$ and $|\eta^{kl}_{m,j}|\leq \gamma$.
Here we fix a $\gamma_0>0$ so that $6^mR^m\gamma_0\leq R'$ as in (27) and let $\tau=|\partial_\eta a(0)|+|\bar\partial_\eta a(0)|$. Then for $\gamma<\gamma_0$
we have easily from (58)
$$|\partial_\sigma a^i|_{E(R,\gamma)}\leq \tau+C_9\|a\|_{C^2(E(R,\gamma_0))}\gamma,$$
which implies $A(R,\gamma)\leq \tau+C_9\|a\|_{C^2(E(R,\gamma_0))}\gamma$. Here $C_9$ is a constant dependent only on $m,n$ and $\|a\|_{C^2(E(R,\gamma_0))}$ denotes
the maximum norm of second derivatives of $a=(a^1, ...,a^n)$ on $E(R,\gamma_0)$. Similarly, we have $B(R,\gamma)\leq \tau+C_9\|a\|_{C^2(E(R,\gamma_0))}\gamma$. On the other hand, we have
\begin{eqnarray}
&&\partial_\sigma a^i(\eta_1,...,\eta_{m-1},\eta_m)-\partial_\sigma a^i(\eta'_1,...,\eta'_{m-1},\eta'_m)\nonumber\\
&=&\int_0^1 \frac{d}{dt}\partial_\sigma a^i(t\eta_1+(1-t)\eta'_1,...,t\eta_{m-1}+(1-t)\eta'_{m-1},t\eta_m+(1-t)\eta'_m)dt\nonumber\\
&=&\int_0^1\{\sum_{j=1}^n\sum_{p=0}^{m}\sum_{k+l=p}\partial_\sigma\partial_{\eta^{kl}_{p,j}} a^i(\cdot)(\eta^{kl}_{p,j}-\eta'^{kl}_{p,j})+\partial_\sigma\bar\partial_{\eta^{kl}_{p,j}} a(\cdot)(\bar\eta^{kl}_{p,j}-\bar\eta'^{kl}_{p,j})\}dt
\end{eqnarray}
It follows from (59) that
$$H_\alpha^A[R, \gamma]\leq C_9 \|a\|_{C^2(E(R,\gamma_0))}\gamma^{1-\alpha}.$$
Similarly, we have
$$H_\alpha^B[R, \gamma]\leq C_9 \|a\|_{C^2(E(R,\gamma_0))}\gamma^{1-\alpha},$$
$$H_1^A[R, \gamma]\leq C_9 \|a\|_{C^2(E(R,\gamma_0))},$$
$$H_1^B[R, \gamma]\leq C_9 \|a\|_{C^2(E(R,\gamma_0))}.$$
We can get from (51), replacing $\gamma$ by $\gamma_0$ in the sum,
$$\delta_2(R,\gamma)\leq C(m,R)\tau+C(R, \gamma_0, m, \alpha, \|a\|_{C^2(E(R,\gamma_0))})(\gamma+\gamma^{1-\alpha}).$$
On the other hand,
$$\delta_1(R,\gamma)=B(R,\gamma)+\gamma2^{m+1}H_1^B[R,\gamma]+\delta_2(R,\gamma),$$
which implies
$$\delta_1(R,\gamma)\leq C(m,R)\tau+C(R, \gamma_0, m, \alpha, \|a\|_{C^2(E(R,\gamma_0))})(\gamma+\gamma^{1-\alpha}).$$
Therefore  we have
$$\delta(R,\gamma)\leq C(R,m,\alpha)\tau+C(R, \gamma_0, m, \alpha, \|a\|_{C^2(E(R,\gamma_0))})(\gamma+\gamma^{1-\alpha})$$
$$\eta(R,\gamma)\leq C(R,m,\alpha)|a(0)|+\gamma (C(R,m,\alpha)\tau+C(R, \gamma_0, m, \alpha, \|a\|_{C^2(E(R,\gamma_0))})(\gamma+\gamma^{1-\alpha})).$$
Now we are ready to determine constants needed to insure existence of solutions.
First we set
$$\tau=\frac{1}{8C(R,m,\alpha)}.$$
Second, choose $\gamma_1$ such that
$$C(R, \gamma_0, m, \alpha, \|a\|_{C^2(E(R,\gamma_0))})(\gamma_1+\gamma_1^{1-\alpha})=\frac{1}{8}$$
Finally, set $\lambda$ to be
$$\lambda=\frac{\gamma_1}{4C(R,m,\alpha)}.$$
Finally, we have $\delta(R,\gamma_1)\leq 1/4, \eta(R,\gamma_1)\leq \gamma'/2$.
To show the existence, let $\psi(z,\bar z)$ be any homogenous polynomial of degree $m$ without term $z^\mu\bar z^\nu$ such that
$\|\psi\|^{(m)}\leq\frac{\gamma_1}{2}$. Consider the equation on $A(R,\gamma_1)$
$$u=\psi+\mathbf{\Theta}(u).$$
By the construction, we have
$$\|\mathbf{\Theta}(f)-\mathbf{\Theta}(g)\|^{(m)}
\leq \frac{1}{4}\|f-g\|^{(m)},$$
$$\|\mathbf{\Theta}(f)\|^{(m)}\leq \frac{\gamma_1}{2}$$
for $f, g\in A(R,\gamma_1)$. Lemma 4.1 applies to conclude the existence of solutions.

\subsection{Proof of Theorem 1.5}
Let $P(m,\epsilon)$ be the set of polynomial maps $\co:\to\co^n$ of degree less or equal to $m-1$ whose coefficients are less than $\epsilon$ in absolute value.
Let $p\in P(m,\epsilon)$. We consider
$$b(z,\tilde u, \mathcal{D}^1\tilde u,..., \mathcal{D}^{m}\tilde u)=a(\tilde u-p(z),\mathcal{D}^1\tilde u-\mathcal{D}^1p(z),..., \mathcal{D}^{m-1}\tilde u-\mathcal{D}^{m-1}p(z), \mathcal{D}^{m}\tilde u).$$
We are looking for solutions for
$$\partial^\mu\bar\partial^\nu \tilde u=b(z,\tilde u, \mathcal{D}^1\tilde u,..., \mathcal{D}^{m}\tilde u)$$
with $\partial^i\bar\partial^j \tilde u(0)=0$ for $i+j\leq m-1$. Then $u=\tilde u+p(z)$ is the solution of the original system with desired property because $\partial^\mu\bar\partial\nu \tilde u=\partial^\mu\bar\partial\nu u$. The proof of Theorem 1.5 will essentially work provided we estimate $C(R,\gamma)$,$H_\alpha^C[R,\gamma]$, and $H_1^C[R,\gamma]$ for $b$, in terms of $\epsilon$ and those constants for $a$. First we modify the equation (27) to
$$6^m R^m\gamma<\frac{\gamma'}{2}$$
so that for small enough $\epsilon$, we still have
$$6^m R^m\gamma+\epsilon<\gamma'.$$
With these choices of $\gamma,\epsilon$, now  we have
$$b(z, \eta_0,\eta_1,...,\eta_{m-1},\eta_m)=a(\eta_0+p(z),\eta_1+\mathcal{D}^1 p(z),...,\eta_{m-1}+\mathcal{D}^{m-1}p(z),\eta_m),$$
which is well defined on $\Omega$.
First we observe that there is constant $C=C(m,R)$ such that if $|z|\leq R$, then
$$|\mathcal{D}^j(p(z))|\leq C(m,R)\epsilon,$$
$$|\nabla (\mathcal{D}^j(p(z)))|\leq C(m,R)\epsilon$$
for $j=0,1,...,m-1.$ Now we are ready to estimate constants for $\delta(R,\gamma)$, $\eta(R,\gamma)$.
$$\frac{\partial b^i}{\partial z}=\sum_{j=1}^n\sum_{q=0}^{m-1}\sum_{k+l=q}\partial_{\eta_{q,j}^{kl}}a^i(\cdot)(\mathcal{D}^q p(z))^{kl}_j+
\sum_{j=1}^n\sum_{k+l=m}\partial_{\eta_{m,j}^{kl}}a^i(\cdot)$$
$$\sum_{j=1}^n\sum_{q=0}^{m-1}\sum_{k+l=q}\partial_{\bar\eta_{q,j}^{kl}}a^i(\cdot)\overline{(\mathcal{D}^q p(z))}^{kl}_j+
\sum_{j=1}^n\sum_{k+l=m}\partial_{\bar\eta_{m,j}^{kl}}a^i(\cdot)$$
It follows that
$$\bigg|\frac{\partial b^i}{\partial z}\bigg|_{E(R,\gamma)}\leq C(m,R)(\epsilon A(R,\gamma+C\epsilon)+B(R,\gamma+C\epsilon)).$$
Therefore $C(R,\gamma)$ for $b$ satisfies
$$C(R,\gamma)\leq C(m,R)(\epsilon A(R,\gamma+C\epsilon)+B(R,\gamma+C\epsilon)).$$
Now we need to estimate H\"older norm in $z$ on $E(R,\gamma)$.
$$\frac{\partial b^i}{\partial z}(z,\eta_0,\eta_1,...,\eta_{m-1},\eta_m)-\frac{\partial b^i}{\partial z}(z',\eta_0,\eta_1,...,\eta_{m-1},\eta_m)$$
$$=\sum_{j=1}^n\sum_{q=0}^{m-1}\sum_{k+l=q}\partial_{\eta_{q,j}^{kl}}a^i(\eta_0+p(z),\eta_1+\mathcal{D}^1 p(z),...,\eta_{m-1}+\mathcal{D}^{m-1}p(z),\eta_m)(\mathcal{D}^q p(z))^{kl}_j$$
$$+
\sum_{j=1}^n\sum_{k+l=m}\partial_{\eta_{m,j}^{kl}}a^i(\eta_0+p(z),\eta_1+\mathcal{D}^1 p(z),...,\eta_{m-1}+\mathcal{D}^{m-1}p(z),\eta_m)$$
$$+\sum_{j=1}^n\sum_{q=0}^{m-1}\sum_{k+l=q}\partial_{\bar\eta_{q,j}^{kl}}a^i(\eta_0+p(z'),\eta_1+\mathcal{D}^1 p(z'),...,\eta_{m-1}
+\mathcal{D}^{m-1}p(z'),\eta_m)\overline{(\mathcal{D}^q p(z'))}^{kl}_j$$
$$+
\sum_{j=1}^n\sum_{k+l=m}\partial_{\bar\eta_{m,j}^{kl}}a^i(\eta_0+p(z),\eta_1+\mathcal{D}^1 p(z),...,\eta_{m-1}+\mathcal{D}^{m-1}p(z),\eta_m)$$
$$-\sum_{j=1}^n\sum_{q=0}^{m-1}\sum_{k+l=q}\partial_{\eta_{q,j}^{kl}}a^i(\eta_0+p(z'),\eta_1+\mathcal{D}^1 p(z'),...,\eta_{m-1}+\mathcal{D}^{m-1}p(z'),\eta_m)(\mathcal{D}^q p(z'))^{kl}_j$$
$$-
\sum_{j=1}^n\sum_{k+l=m}\partial_{\eta_{m,j}^{kl}}a^i(\eta_0+p(z'),\eta_1+\mathcal{D}^1 p(z'),...,\eta_{m-1}+\mathcal{D}^{m-1}p(z'),\eta_m)$$
$$-\sum_{j=1}^n\sum_{q=0}^{m-1}\sum_{k+l=q}\partial_{\bar\eta_{q,j}^{kl}}a^i(\eta_0+p(z'),\eta_1+\mathcal{D}^1 p(z'),...,\eta_{m-1}
+\mathcal{D}^{m-1}p(z'),\eta_m)\overline{(\mathcal{D}^q p(z'))}^{kl}_j$$
$$-
\sum_{j=1}^n\sum_{k+l=m}\partial_{\bar\eta_{m,j}^{kl}}a^i(\eta_0+p(z'),\eta_1+\mathcal{D}^1 p(z'),...,\eta_{m-1}+\mathcal{D}^{m-1}p(z'),\eta_m)$$
We have
$$|\mathcal{D}^jp(z)-\mathcal{D}^jp(z')|\leq C(m,R,\alpha)\epsilon|z-z'|$$
$$H_\alpha\bigg[\frac{\partial b^i}{\partial z}\bigg]_{E(R,\gamma)}\leq \epsilon^\alpha C(m,R,\alpha)(H_\alpha^A[R,\gamma+C\epsilon]+A(R,\gamma+C\epsilon)+
H_\alpha^B[R,\gamma+C\epsilon]+B(R,\gamma+C\epsilon).$$
Therefore $H_\alpha^C[R,\gamma]$ for $b$ satisfies
$$H_\alpha^C[R,\gamma]\leq \epsilon^\alpha C(m,R,\alpha)(H_\alpha^A[R,\gamma+C\epsilon]+A(R,\gamma+C\epsilon)+
H_\alpha^B[R,\gamma+C\epsilon]+B(R,\gamma+C\epsilon).$$
Similarly,
$$H_1\bigg[\frac{\partial b^i}{\partial z}\bigg]_{E(R,\gamma)}\leq \epsilon C(m,R)(H_1^A[R,\gamma+C\epsilon]+A(R,\gamma+C\epsilon)+
H_1^B[R,\gamma+C\epsilon]+B(R,\gamma+C\epsilon)$$
Therefore $H_1^C[R,\gamma]$ for $b$ satisfies
$$H_1^C[R,\gamma]\leq \epsilon C(m,R)(H_1^A[R,\gamma+C\epsilon]+A(R,\gamma+C\epsilon)+
H_1^B[R,\gamma+C\epsilon]+B(R,\gamma+C\epsilon)$$
For all other constants of $b$, we only need replace by those of $a$ on $E(R, \gamma+C\epsilon)$. For $b$, $\delta(R,\gamma)$ is virtually unchanged and $\eta(R,\gamma)$ has new constants obtained above. Then all proof of Theorem 1.4 will go though with little modification. We omit the details.

\section{Real systems-Proof of Theorem 1.6,1.7}
In order to obtain real solutions of systems, we need to confine to the case $\mu=\nu$. We also have to work on Banach space over $\re$. To do so, we
define
$$RC_0^{k+\alpha}(D)=\{f: f\in C_0^{k+\alpha}(D) \mbox{ and } f \mbox{ is real-valued }\}$$
This is Banach space with norm $\|\cdot\cdot\cdot\|^{(k)}$ just as defined for complex valued functions.
We notice that $\Delta^m u=4^m\partial^m\bar\partial^m u$ and the operator $\Xi=\overline T^m T^m+T^m\overline T^m$ is real-valued; that is, if $f$ is real-valued so is $\Xi(f)$.
Let
$$\omega^i(f)=\overline T^m T^mA^i(x,y, \nabla f, \nabla^2 f,...,\nabla^{2m-1}f)$$

In order to prove the existence, we consider equation
$$u^i=\psi+\Re \Theta^i(u)=\psi+\frac{1}{2}\{\Theta^i(u)+\overline{\Theta^i(u)}\}$$
where
$$\Theta^i(f)=\omega^i(f)-\sum_{l=0}^{2m-1}\frac{1}{l!}\sum_{i+j=l}[\partial^i\bar\partial^j\omega^i(f)](0)\zeta^i\bar\zeta^j
-\frac{1}{(2m)!}\sum_{i+j=2m}[\partial^i\bar\partial^j\omega^i(f)](0)\zeta^i\bar\zeta^j+\frac{1}{m!^2}[\partial^m\bar\partial^m\omega^i(f)](0)\zeta^m\bar\zeta^m$$
for $f\in [RC_0^{k+\alpha}(D)]^n$ and where $\psi$ is real valued  and satisfies $\Delta^m \psi=0,$ and $\partial^i\bar\partial^j \psi(0)=0$ for $i+j\leq 2m-1$. Therefore we define
$$\Re \Theta: [RC_0^{k+\alpha}(D)]^n\to [RC_0^{k+\alpha}(D)]^n,$$
$$\Re \Theta(f)=(\Re \Theta^1(f), ..., \Re \Theta^n(f)).$$
With $\mathbf{A}(R,\gamma)=\{f\in [RC_0^{k+\alpha}(D)]^n|\|f\|^{(2m)}\leq \gamma\}$, we can apply the same arguments as in the proof of Theorem
to get desired results.
\section{Holomorphic system}
The key observation is the following lemma. Let $$ HC^{k+\alpha}(D)=\{f: f\in C^{k+\alpha}(D) \mbox{ and holomorphic in } \mathrm{Int}(D)\}.$$
\begin{lem}
The operator $\overline T$ maps $HC^{k+\alpha}(D)$ into $HC^{k+1+\alpha}(D).$
\end{lem}
\begin{proof}
By Lemma 3.4, we have $T(\bar f)$ is anti-holomorphic if $f$ is holomorphic. By the definition, $\overline T(f)=\overline {T(\bar f)}$. The proof follows.

\end{proof}
In order to prove the theorem, we first notice that uniqueness is a result of unique continuation property of holomorphic functions. We let
$$HC_0^{k+\alpha}(D)=\{f: f\in C_0^{k+\alpha}(D) \mbox{ and holomorphic in } \mathrm{Int}(D)\}.$$
We first note that
equipped with $\|\cdot\cdot\cdot\|^{(k)}$,$HC_0^{k+\alpha}(D)$ is a Banach subspace of $C_0^{k+\alpha}(D)$. We will work on this subspace to prove existence
of solutions. Indeed, we define for $i=1,...,n$
$$\omega^i(f)=\overline T^m H(z,f, \partial f, ..., \partial^{m-1} f)$$
where $f\in  HC_0^{m+\alpha}(D)$. Since $H$ is holomorphic map, so $\omega^i(f)$ is holomorphic by Lemma ? and $\omega^i(f)\in HC^{m+\alpha}(D).$
We consider the equation
$$f^i=\Theta^i(f)$$
where
$$\Theta^i(f)=\overline T^m H(z,f, \partial f, ..., \partial^{m-1} f)-\sum_{k=0}^{m-1}\frac{1}{k!}[\partial^k \omega^i(f)](0).$$
By the construction, we have $\Theta^i(f)\in HC_0^{m+\alpha}(D)$. Then the following map is well defined
$$\Theta:HC_0^{m+\alpha}(D)\to HC_0^{m+\alpha}(D)$$
$$\Theta(f)=(\Theta^1(f),...,\Theta^n(f))$$
Now given $\gamma>0$, we define $\mathbf{A}(R,\gamma)=\{\|\partial^m f\|\leq \gamma\}$. We notice that the norm $\|f\|^{(m)}$ reduces to
$\|\partial^m f\|$ since $f$ is holomorphic. With these preparations, all proofs would apply to holomorphic situations. We would not repeat the
arguments. We note that all solutions are holomorphic disks in $\co^n$ that satisfy the differential system as in Theorem 1.12. If the system has a geometric information
then they become geometric objects.
\section{Kobayashi metric on a Riemannian manifold}
\subsection{Proof of Theorem B}
Since it is local, we choose charts at $z_0$ and $p$ respectively, as $(D, z)$, $(V, w)$. Let $N$ be given with metric tensor $(g_{ij})$. Then a map $f:D\to N$ of class $C^2$ is harmonic iff
$$\frac{\partial^2 f^i}{\partial z\partial \bar z}+\Gamma^i_{jk}(f(z))\frac{\partial f^j}{\partial z}\frac{\partial^k}{\partial\bar z}=0 \mbox{ for } i=1,..., \mathrm{dim}N.$$
Here $\Gamma^i_{jk}$ is the Christoffel symbols of $N$. The proof follows from Theorem 1.5 immediately.
\subsection{Definition of Kobayashi metric}
Here we introduce the notion of Kobayaashi metric on the tangent bundle of a Riemannian manifold following the one on complex manifolds by holomorphic curves.
\begin{defn} Let $N$ be a Riemannian manifold of class $C^3$ with metric $g$. Let $p\in N$, and $v\in T_pN$. We define the Kobayashi metric on $TN$ as follows
$$K_N(p,v)=\inf\{\frac{1}{R}: f:D(R):\to N \mbox{ harmonic map of class } C^2 \mbox{ such that } f(0)=p, df(0)(\frac{\partial}{\partial x})=v\},$$
where $D(R)=\{z\in\co: |z|<R\}$ with $z=x+iy$.
\end{defn}
The first simple result is
\begin{prop}
Let $N$ be a smooth Riemannian manifold. Then the Kobayahsi metric is well defined on the tangent bundle $TN$.
\end{prop}
We hope to make further study of Kobayashi metric in a future paper.

\bigskip
\appendix{Appendix}
\begin{lem}
\begin{eqnarray}
Let
\partial
_x^\mu&=&(\partial+\bar\partial)^\mu=\sum_{k=0}^\mu {\mu\choose k}\partial^k\bar\partial^{\mu-k},\nonumber\\
\partial_y^\nu&=&i^\nu(\partial-\bar\partial)^\nu=i^\nu\sum_{l=0}^\nu (-1)^{\nu-l}{\nu\choose l}\partial^l\bar\partial^{\nu-l}.\nonumber
\end{eqnarray}
It holds that for $\mu,\nu\geq 0$,
$$\partial_x^\mu\partial_y^\nu=i^\nu\sum_{j=0}^{\mu+\nu}A_j\partial^{j}\bar\partial^{\mu+\nu-j}$$
where
$$A_j=\sum_{l={\max\{0,j-\mu\}}}^{\min\{\nu,j\}}{\mu\choose {j-l}}{\nu\choose l}(-1)^{\nu-l}.$$
\end{lem}
\begin{proof} First we convert partial derivatives to complex partial derivatives. Indeed,
\begin{eqnarray}
\partial_x^\mu\partial_y^\nu&=&i^\nu\sum_{k=0}^\mu {\mu\choose k}\partial^k\bar\partial^{\mu-k}\sum_{l=0}^\nu (-1)^{\nu-l}{\nu\choose l}\partial^l\bar\partial^{\nu-l}\nonumber\\
&=&i^\nu\sum_{k=0}^\mu\sum_{l=0}^\nu{\mu\choose k}{\nu\choose l}(-1)^{\nu-l}\partial^{k+l}\bar\partial^{\mu+\nu-k-l}\nonumber\\
&=&i^\nu\sum_{j=0}^{\mu+\nu}A_j\partial^{j}\bar\partial^{\mu+\nu-j}
\end{eqnarray}
Now we prove the formula by induction. Let $m=\mu+\nu$. If $m=1$, then either $\mu=1$ or $\nu=1$. The formula is obvious.
Now assume the formula is true for $m$. Now we want to show it is also true for $\mu+\nu=m+1$. Below we assume $\mu+\nu=m+1$. If $\mu=0$ then $\nu=m+1$, and the formula is obvious.
So we assume below that $\mu\geq 1$. Thus we have $\mu-1+\nu=m$. We have
$$\partial_x^\mu\partial_y^\nu=\partial_x\partial_x^{\mu-1}\partial_y^\nu=i^\nu\partial_x
\sum_{j=0}^{\mu+\nu-1}B_j\partial^{j}\bar\partial^{\mu+\nu-1-j}=i^\nu\partial_x
\sum_{j=0}^{m}B_j\partial^{j}\bar\partial^{m-j}$$
where
$$B_j=\sum_{l={\max\{0,j-\mu+1\}}}^{\min\{\nu,j\}}{\mu-1\choose {j-l}}{\nu\choose l}(-1)^{\nu-l}.$$
Therefore,
$$\partial_x^\mu\partial_y^\nu=i^\nu\{
B_0\bar\partial^{m+1}+\sum_{j=1}^{m}(B_{j-1}+B_j)\partial^{j}\bar\partial^{m+1-j}+B_m\partial^{m+1}\}$$
$$=i^\nu\sum_{j=1}^{m+1}\tilde A_j\partial^{j}\bar\partial^{m+1-j}$$
where
$\tilde A_0=B_0,\tilde A_j=B_{j-1}+B_j$ for $ j= 1,...,m,$ and $\tilde A_{m+1}=B_m$. Now it suffices to show that for $0\leq j\leq m+1$
$$\tilde A_j=\sum_{l={\max\{0,j-\mu\}}}^{\min\{\nu,j\}}{\mu\choose {j-l}}{\nu\choose l}(-1)^{\nu-l}.$$
It is easy to see that this true for $j=0$, and $j=m+1$. To show it is also true for $j=1,...,m$, we have to show
$$\sum_{l={\max\{0,j-\mu\}}}^{\min\{\nu,j-1\}}{\mu-1\choose {j-1-l}}{\nu\choose l}(-1)^{\nu-l}+
\sum_{l={\max\{0,j-\mu+1\}}}^{\min\{\nu,j\}}{\mu-1\choose {j-l}}{\nu\choose l}(-1)^{\nu-l}$$
$$=\sum_{l={\max\{0,j-\mu\}}}^{\min\{\nu,j\}}{\mu\choose {j-l}}{\nu\choose l}(-1)^{\nu-l}$$
Now we consider two cases:
If $j-\mu+1>0$, then $j-\mu\geq 0$. The identity becomes
$$\sum_{l=j-\mu}^{\min\{\nu,j-1\}}{\mu-1\choose {j-1-l}}{\nu\choose l}(-1)^{\nu-l}+
\sum_{l={j-\mu+1}}^{\min\{\nu,j\}}{\mu-1\choose {j-l}}{\nu\choose l}(-1)^{\nu-l}$$
$$=\sum_{l=j-\mu}^{\min\{\nu,j\}}{\mu\choose {j-l}}{\nu\choose l}(-1)^{\nu-l}$$
If $j-\mu+1\leq 0$, then $j-\mu\leq -1$. The identity becomes
$$\sum_{l=0}^{\min\{\nu,j-1\}}{\mu-1\choose {j-1-l}}{\nu\choose l}(-1)^{\nu-l}+
\sum_{l=0}^{\min\{\nu,j\}}{\mu-1\choose {j-l}}{\nu\choose l}(-1)^{\nu-l}$$
$$=\sum_{l=0}^{\min\{\nu,j\}}{\mu\choose {j-l}}{\nu\choose l}(-1)^{\nu-l}$$
Finally the both identities can be verified by considering three cases: $\nu<j-1, \nu=j-1,$ and $\nu>j-1$ along with ${n+1\choose k+1}={n\choose k}+{n\choose k+1}$.
\end{proof}

{\bf Acknowledgment:} The research project of this paper was initiated while I was visiting at Huzhou Teachers College per invitation of
Prof. Zhangjian Hu  and Taishun Liu. The progress was made during my one-month lectures at Nanjing Normal University invited by Professor Prof. Hauihui Chen.
I also benefited from visits to Nanjing University, University of Science and Technology of China and Zhejing Normal University by invitations of
Prof. Liangwen Liao and Guanbin Ren, and Yang Liu. Hereby, I would like to take this opportunity to thank these host institutions for their warm hospitality.

Last, but not least, I would like to thank Prof. Zhihua Chen for his support over the years for my numerous visits to Tongji University.

\bigskip
Department of Mathematical Sciences

Indiana University - Purdue University Fort Wayne

Fort Wayne, IN 46805-1499, USA.

pan@ipfw.edu
\end{document}